\documentclass[reqno,11pt]{amsart}

\usepackage{amssymb}
\usepackage[all]{xy}
\usepackage{caption}
\usepackage[pdftex]{hyperref}
\usepackage{graphicx}
\usepackage{array}
\usepackage{physics}
\usepackage{color}
\usepackage{stmaryrd}
\usepackage{xypic}
\usepackage{bm}
\usepackage{comment}
\usepackage{mathtools}
\usepackage{tikz}
\usetikzlibrary{knots, hobby}
\allowdisplaybreaks

\newtheorem{thmIntro}{Theorem}
\newtheorem{thm}{Theorem}[section]
\newtheorem{lem}[thm]{Lemma}
\newtheorem{prop}[thm]{Proposition}

\newtheorem{cor}[thm]{Corollary}
\theoremstyle{definition}
\newtheorem{defn}[thm]{Definition}
\newtheorem{ex}[thm]{Example}
\newtheorem{rem}[thm]{Remark}
\newtheorem*{conj*}{Conjecture}

\newcommand{\CC}{{\mathbb C}}
\newcommand{\ZZ}{{\mathbb Z}}
\newcommand{\RR}{{\mathbb R}}

\newcommand{\cT}{{\mathcal{T}}}
\newcommand{\R}{{\mathcal{R}}}
\newcommand{\FQ}{\textbf{FQ}}

\newcommand{\cA}{{\mathcal{A}}}

\newcommand{\Free}[1]{\mathbf{F}_{#1}}
\newcommand{\SVB}{\textbf{\textup{SVectB}}}

\newcommand{\Qdl}{\textbf{\textup{Qdl}}}

\newcommand{\AQ}[1]{{\textup{AQ}_{#1}}}

\newcommand{\AGL}[1]{{\textup{AGL}_{#1}}}

\newcommand{\aglC}{\AGL{1}(\mathbb{C})}
\newcommand{\agl}{\AGL{1}}

\DeclareMathOperator{\id}{Id}
\DeclareMathOperator{\Id}{Id}

\DeclareMathOperator{\Hom}{Hom\,}           

\DeclareMathOperator{\Gr}{Gr}
    
\DeclareMathOperator{\diag}{diag}

\DeclareMathOperator{\Conj}{Conj}
\DeclareMathOperator{\Span}{Span}
\DeclareMathOperator{\Vect}{\mathbf{Vect}}
\DeclareMathOperator{\As}{As}
\DeclareMathOperator{\Eq}{Eq}
\DeclareMathOperator{\sign}{sgn}
\DeclareMathOperator{\coevmap}{coev}
\DeclareMathOperator{\evalmap}{ev}

\newcommand{\capright}{
  \begin{tikzpicture}[scale=0.35,baseline=0.5ex]
    \coordinate (base) at (0,0);
    \draw[thick,->]
      (0,0.1) .. controls (0.1,1.3) and (1,1.3) .. (1,0.1);
  \end{tikzpicture}
}
\newcommand{\capleft}{
  \begin{tikzpicture}[scale=0.35,baseline=0.5ex]
    \coordinate (base) at (0,0);
    \draw[thick,<-]
      (0,0.1) .. controls (0.1,1.3) and (1,1.3) .. (1,0.1);
  \end{tikzpicture}
}

\newcommand{\cupleft}{%
  \begin{tikzpicture}[scale=0.35,baseline=0.5ex]
    \draw[thick,<-]
      (0,1) .. controls (0,-0.2) and (1,-0.2) .. (1,1);
  \end{tikzpicture}%
}

\newcommand{\cupright}{%
  \begin{tikzpicture}[scale=0.35,baseline=0.5ex]
    \draw[thick,->]
      (0,1) .. controls (0,-0.2) and (1,-0.2) .. (1,1);
  \end{tikzpicture}%
}

\newcommand{\crossingXminusthree}{
  \begin{tikzpicture}[scale=0.4,baseline=0.5ex]
    \coordinate (base) at (0,-0.5);
    \draw[thick,->] (0,1) -- (1,0);
    \draw[white,line width=4pt] (0,0) -- (1,1);
    \draw[thick,->] (0,0) -- (1,1);
  \end{tikzpicture}
}

\newcommand{\crossingXplus}{
  \begin{tikzpicture}[scale=0.4,baseline=0.5ex]
    \coordinate (base) at (0,-0.5);
    \draw[thick,->] (1,0) -- (0,1);
    \draw[white,line width=4pt] (0,0) -- (1,1);
    \draw[thick,->] (0,0) -- (1,1);
  \end{tikzpicture}
}

\newcommand{\crossingXminusone}{
  \begin{tikzpicture}[scale=0.4,baseline=0.5ex]
    \coordinate (base) at (0,-0.5);
    \draw[thick,->] (1,0) -- (0,1);
    \draw[white,line width=4pt] (0,0) -- (1,1);
    \draw[thick,->] (1,1) -- (0,0);
  \end{tikzpicture}
}

\newcommand{\crossingXplustwo}{
  \begin{tikzpicture}[scale=0.4,baseline=0.5ex]
    \coordinate (base) at (0,-0.5); 
    \draw[thick,->] (0,1) -- (1,0);
    \draw[white,line width=4pt] (0,0) -- (1,1);
    \draw[thick,->] (1,1) -- (0,0);
  \end{tikzpicture}
}

\newcommand{\crossingXminus}{
  \begin{tikzpicture}[scale=0.4,baseline=0.5ex]
    \coordinate (base) at (0,-0.5);
    \draw[thick,->] (0,0) -- (1,1);
    \draw[white,line width=4pt] (0,1) -- (1,0);
    \draw[thick,->] (1,0) -- (0,1);
  \end{tikzpicture}
}

\newcommand{\crossingXplusthree}{
  \begin{tikzpicture}[scale=0.4,baseline=0.5ex]
    \coordinate (base) at (0,-0.5); 
    \draw[thick,->] (0,0) -- (1,1);
    \draw[white,line width=4pt] (0,1) -- (1,0);
    \draw[thick,->] (0,1) -- (1,0);
  \end{tikzpicture}
}

\newcommand{\crossingXminustwo}{
  \begin{tikzpicture}[scale=0.4,baseline=0.5ex]
    \coordinate (base) at (0,-0.5);
    \draw[thick,->] (1,1) -- (0,0);
    \draw[white,line width=4pt] (0,1) -- (1,0);
    \draw[thick,->] (0,1) -- (1,0);
  \end{tikzpicture}
}

\newcommand{\crossingXplusone}{
  \begin{tikzpicture}[scale=0.4,baseline=0.5ex]
    \coordinate (base) at (0,-0.5); 
    \draw[thick,->] (1,1) -- (0,0);
    \draw[white,line width=4pt] (0,1) -- (1,0);
    \draw[thick,->] (1,0) -- (0,1);
  \end{tikzpicture}
}

\subjclass[2020]{Primary: 57K14. 
Secondary: 14M35, 
57K12, 
18M20. 
}
\keywords{Alexander polynomial, representation variety, Topological Quantum Field Theory, knot quandle, rational knot, pretzel knot.}

\title{Affine representations of rational and pretzel knots}

\author{Javier Mart\'{\i}nez}
\address{Departamento de Matem\'atica Aplicada a la Ingeniería Industrial, Escuela Técnica Superior de Ingeniería y Diseño Industrial, Universidad Politécnica de Madrid, Ronda de Valencia 3, 28012 Madrid Spain.}
\email{javier.martinezmartinez@upm.es}

\begin{document}

\maketitle

\begin{abstract}
We construct and study representations of rational and pretzel tangle and knot groups into the affine group $\mathrm{AGL}(1,\mathbb{C})$, via a TQFT that is valued in the category of spans of singular vector bundles over $\mathbb{C}^{\ast}$. For these families, we derive closed-form expressions for their Alexander polynomials and establish bounds on their zeros. Finally, we specialize the functor at $t=-1$ and analyze colorings of rational tangles in terms of spans of complex vector spaces.
\end{abstract}

\section{Introduction}\label{sec:introduction}

The Alexander polynomial, introduced nearly a century ago, continues to be a fundamental knot invariant. Classical approaches introduce it via the Seifert matrix, using Fox calculus, or through a presentation of the knot group and its abelianization \cite{BurdeZieschangbook}, all encoding the properties of the first homology group of the infinite cyclic cover of the knot complement. In addition, several generalizations have been developed in the last decades, leading to interesting and fruitful theories \cite{Lin,surveytwisted}. Other powerful knot invariants, such as knot Floer homology \cite{knotfloer1}, recover it as a graded Euler characteristic, showing once again its importance and ubiquity.

There are several approaches to categorifying the Alexander polynomial, many of which give rise to   quantum field theories for Alexander invariants \cite{frohman1992alexander, bigelow2015alexander, florens2016functorial, cimasoni2005lagrangian, cimasoniconway}. From a representation-theoretic perspective, \cite{AGL1Alexanderinvariants} established their deep connection with complex affine representations of the knot group, a relation first observed by De Rham \cite{deRham} and Burde \cite{Burde}, and later studied from a sheaf-theoretic perspective by Hironaka \cite{hironaka1997alexander}. For any knot $K$, the space of $\aglC$-representations of the knot group $\Gamma_K$ can be interpreted as a singular vector bundle (a geometric space associated to a coherent sheaf), if one considers the regular map between representation varieties
\begin{equation} \label{eqn:Alexfibration}
\pi_K: \R_{\aglC}(\Gamma_K) \longrightarrow \R_{\CC^{\ast}}(\Gamma_K) = \CC^*
\end{equation}
induced by the projection onto the diagonal entry $\aglC\rightarrow \CC^*$. The fibre of this map over any $t\in \CC^*$ can be identified with
$$
\R_{\AQ{t}}(Q_K) = \Hom_\Qdl(Q_K, \AQ{t}),
$$
the set of quandle representations of the knot quandle $Q_K$ into the Alexander quandle $\AQ{t}$. These ideas extend naturally from knots to tangles, leading to a Topological Quantum Field Theory, developed in \cite{AGL1Alexanderinvariants}, that encodes the local and global properties of the Alexander module in terms of representations:

\begin{thmIntro}\label{thm-intro:TQFT-global}
There exists a braided monoidal functor
$$
\cA: \cT \longrightarrow \Span(\SVB),
$$
computing the Alexander module of knots.
\end{thmIntro}

\begin{thmIntro}\label{thm-intro:TQFT-fibre}
For any $t\in \CC^{\ast}$, there exist a braided monoidal functor $\cA_t$,
$$
\cA_t: \cT\longrightarrow\Span(\Vect_{\CC}),
$$
which computes the quandle $\AQ{t}$-representation variety.
\end{thmIntro}

We give a brief overview of the main ideas of this construction in Section \ref{sec:agl1repsandAlexanderinvariants}, but let us emphasize that the target category of Theorem \ref{thm-intro:TQFT-fibre} makes possible to apply linear algebra (within the category of spans of complex vector spaces) to compute Alexander invariants of tangles, a feature already present in the subcategory of braids through the Burau representation. 

This manuscript applies Theorems \ref{thm-intro:TQFT-global} and \ref{thm-intro:TQFT-fibre} to two classical  and well-studied families in knot theory, namely rational knots and pretzel knots (and their associated tangles), in analogy with other TQFTs \cite{cimasoni2006lagrangian}. As these families are well studied, our approach recovers several classical results from a different perspective and also produces new formulas and insights.
 First, we study the $\aglC$-representations of elementary oriented 2-tangles and their compositions, which give rise to codirected and alternating braidings, and we express them in terms of quantum integers, $[n]_t=\frac{1-t^n}{1-t}$. A careful analysis of the way these tangles are glued to form rational and pretzel knots and tangles leads us to closed-form expressions for their Alexander modules. Focusing on their Alexander polynomials, we obtain:

\begin{thmIntro}
Let $K$ be a rational knot, such that $p/q=[2k_1,\ldots,2k_{2n}]$. Then
\begin{equation} \label{eqnintro:ContinuantRationalAlexformula}
\Delta_K(t)=K_{2n}([2]_{-t}k_1,-[2]_{-t^{-1}}k_2,\ldots,[2]_{-t}k_{2n-1},-[2]_{-t^{-1}}k_{2n}),
\end{equation}
where $K_{2n}(x_1,\ldots,x_{2n})$ is the continuant 
\begin{equation}
K_{2n}(x_1,\ldots, x_{2n})=\operatorname{det}\left(\begin{array}{ccccc}
x_1 & 1 & \cdots & & 0 \\
1 & x_2 & 1 & & \\
\vdots & 1 & x_3 & \ddots & \\
& & \ddots & \ddots & 1 \\
0 & & & 1 & x_{2n}
\end{array}\right).
\end{equation}
\end{thmIntro}
For pretzel knots, we derive the following formulas in the even and odd cases:

\begin{thmIntro}
Let $K=P(q_1,\ldots,q_n)$ be a pretzel knot. The Alexander polynomial $\Delta_K(t)$ equals:
\begin{itemize}
\item When $n$ and $q_i$ are odd,
$$
\Delta_K(t) =\frac{1}{2^{n-1}}\sum_{k\geq 0} (t-1)^{2k}(t+1)^{n-1-2k}\sigma_{2k}(q_1,\ldots,q_n),
$$
where $\sigma_i(x_1,\ldots,x_n)$ denotes the $i$-th elementary symmetric polynomial.
\item When $q_1$ and $n=2p$ are even and $q_i$ are odd for $i\geq 2$, 
$$
\Delta_K(t)=\left( \prod_{i=1}^{p} [-q_{2i-1}]_{-t} \prod_{i=1}^{p}[q_{2i}]_{-t}\right) \left( \sum_{i=1}^{p}\frac{1}{[q_{2i}]_{-t}}-\sum_{i=1}^{p} \frac{1}{[-q_{2i-1}]_{-t}} \right). 
$$
\item When $q_1$ is even and $n=2p+1$ and $q_i$ are odd for $i\geq 2$,
$$
\Delta_K(t)  =\left( \frac{q_1}{2}[-2]_{-t} \prod_{i=2}^{p+1} [-q_{2i-1}]_{-t} \prod_{i=1}^{p}[q_{2i}]_{-t}\right)  \left( \sum_{i=1}^{p} \frac{1}{[q_{2i}]_{-t}}-\sum_{i=2}^{p+1}\frac{1}{[-q_{2i-1}]_{-t}}+\frac{2}{q_1}\frac{1}{t[-2]_{-t}} \right).  
$$
\end{itemize}
\end{thmIntro}

Closed-form expressions for the Alexander polynomial of these families have previously appeared in the literature. For instance, a continuant expression for the Conway polynomial of rational knots can be obtained from the skein relation \cite{Cromwell} or via a Seifert matrix \cite{LyubichMurasugi}, and closed formulas for the polynomials of general pretzel knots have recently been derived using the skein relation \cite{BelousovAlexander}, coinciding with our results in the odd case. Particular cases have also been studied (see \cite{PretzelAlex1, PretzelAlex2}), and a connection with Fibonacci polynomials in this context was established in \cite{KoseleffPecker}. In contrast, our method, which arises from our TQFT framework, determines the full Alexander module, thus revealing additional structural features. For example, the pretzel formulas are naturally deduced from a continuant recurrence that reflects the topological construction of pretzel tangles, leading via the functor $\cA_t$ to a cyclic tridiagonal matrix that matches the cyclic nature of pretzel knots.

Using these formulas, we are able to analyze the zeros of the Alexander polynomial in these families. We recall the following conjecture:

\begin{conj*}[Hoste]
Let $K$ be an alternating knot, and let $t \in \CC$ be such that $\Delta_K(t)=0$. Then $\Re(t) > -1$.
\end{conj*}

The conjecture was proved by Ishikawa to be true for rational knots \cite{IshikawaHoste}, and false for general knots \cite{hostefalse, IshikawaMontesinos}, after intense research regarding the location of the zeros of the Alexander polynomial \cite{StoimenovHoste,StoimenowHoste2, HosteSurvey}. Our results provide alternative proofs of this conjecture for pretzel knots that admit an alternating diagram, providing stronger bounds on the zeros of the Alexander polynomial in some of these pretzel families. Specifically, we prove the following:

\begin{thmIntro}
Let $K$ be the even alternating pretzel knot $K(q_1,q_2,\ldots,q_n)$, where $q_1$ is even and $q_i$ are positive odd integers for $i\geq 2$. Let $t \in \CC$ be such that  $\Delta_K(t)=0$. Then:
\begin{itemize}
\item If $n$ is even, $\abs{t}=1$.
\item If $n$ is odd,  $\Re(t) > -1$.
\end{itemize}
\end{thmIntro}

\begin{thmIntro}
Let $K$ be the odd alternating pretzel knot $K(q_1,\ldots,q_n)$, where $n$ and $q_i$ are positive odd integers. Let $t \in \CC$ be such that $\Delta_K(t)=0$. Then $\abs{t}=1$.
\end{thmIntro}

Finally, we show that the specialization of the functor $\cA_t$ at $t=-1$ recovers a classical result in knot theory, namely the classification of rational tangles. This demonstrates that the TQFT developed in \cite{AGL1Alexanderinvariants} complements existing approaches in the literature, such as Conway’s original method via the Alexander-Conway polynomial \cite{conwayrational}, or the topological and arithmetic proofs via colorings due to Kauffman and Lambropoulou \cite{kauffmanrationaltangles}. Our perspective emphasizes how the span category, together with the rotation of spans induced by the rotation of 2-tangles, provides a natural framework to ``linearize’’ tangles which are not braids, taking into account the image of the embedding $\Vect_{\CC} \hookrightarrow \Span(\Vect_{\CC})$. Concretely, we establish the following:  

\begin{thmIntro}
For any rational tangle $T_{p/q}$, the specialization $\cA_{-1}(T)$ is the basic span determined by the linear map  
$$ 
f^{-q/p}:\CC^2 \longrightarrow \CC^2, \qquad f^{-q/p}=\id-\tfrac{q}{p}h,  
$$  
where $h$ is the nilpotent map $h:\CC^2\longrightarrow \CC^2$ defined as $h(a_1,a_2)=(a_1-a_2,a_1-a_2)$.  
\end{thmIntro}

This leads to a reinterpretation of Kauffman--Lambropoulou’s proof of the classification of rational tangles via colorings in terms of spans and the geometry of the complex Grassmannian $\Gr(2,4)$. In particular, we prove:  

\begin{thmIntro}
The set  
\[
\cA_{-1}^{rat}:=\{\cA_{-1}(T) \mid T \text{ is a rational tangle}\}  
\]  
defines a smooth rational curve inside the complex Grassmannian $\Gr(2,4)$. Moreover, the fraction associated to any coloring $(v_1,v_2,v_3,v_4)$ of a rational tangle $T$,  
\[
\frac{p}{q}=\frac{v_4-v_3}{v_4-v_2},  
\]  
coincides with the rational slope of the point $\cA_{-1}(T)$ on $\cA_{-1}^{rat}\cong \mathbb{P}^1(\CC)$.  
\end{thmIntro}

\subsection*{Structure of the manuscript} Section \ref{sec:agl1repsandAlexanderinvariants} provides a brief overview of the construction of the functors $\cA$ and $\cA_t$, together with explicit generating sets for the relevant groups and categories, which will be needed for computations in later sections. Section \ref{sec:2-tangles} analyzes the action of $\cA_t$ on elementary tangles, with emphasis on alternating and codirected braidings, which serve as the building blocks in the decompositions of 2-bridge knots and pretzel knots. These two families are treated in Sections \ref{sec:2-bridgeknots} and \ref{sec:pretzelknots}, where explicit formulas for their Alexander polynomials are derived and their zeros are analyzed. Finally, Section \ref{sec:functorA-1} is devoted to the specialization $\cA_{-1}$ and its connection with rational tangles.

\subsection*{Acknowledgements}

The author thanks \'{A}ngel Gonz\'{a}lez Prieto and Vicente Mu\~{n}oz for many valuable discussions on this work, and Joan Porti and Pedro M. Gonz\'{a}lez Manch\'{o}n for suggestions and conversations regarding pretzel knots. 

\section{$\aglC$-representations and Alexander invariants} \label{sec:agl1repsandAlexanderinvariants} 
In this section we review the construction of the functors $\cA_t$ and $\cA$ and state some of their properties that are used throughout the paper. We begin by recalling the relevant categories, then outline the construction of these functors and, after fixing generators, carry out explicit computations. For full details, see \cite{AGL1Alexanderinvariants}.

\subsubsection*{The tangle category} Let $\cT$ be the category of tangles, whose objects are sequences of $\pm$ signs, possibly empty, which we geometrically realize as a set of evenly spaced points in $\mathbb{R}^2$ with integer coordinates. An \emph{oriented tangle} $T$ is defined as the union any finite number of disjoint polygonal arcs in $\mathbb{R}^2\times [0,1]$ such that the boundary satisfies the condition 
$$
\partial T=T \cap\left(\mathbb{R}^2 \times\{0,1\}\right)=([k] \times\{0\} \times\{0\}) \cup([\ell] \times\{0\} \times\{1\}),
$$
where $[n]=\lbrace 1, \ldots, n\rbrace$ and $[0]=\emptyset$. The isotopy classes of these oriented tangles define the morphisms of the category $\mathcal{T}$, between the \emph{source} of $T$, that is, $\partial T\cap (\mathbb{R}^2\times \lbrace 0 \rbrace)$, and the \emph{target} of $T$, $\partial T\cap (\mathbb{R}^2\times \lbrace 1 \rbrace)$. The composition of two tangles $T$ and $T'$, when defined, is the isotopy class of the tangle that is obtained by placing $T'$ on top of $T$, properly rescaled to produce a tangle in $\mathbb{R}^2\times [0,1]$. Concatenation of tangles endows the category $\cT$ with a non-symmetric braided monoidal structure. We recall that the category $\cT$ can be presented in terms of generators and relations \cite{Kasselquantumgroups}. Explicitly, it is generated by the six elementary morphisms
\begin{equation}\label{eqn:generatorstanglecat}
\cupright, \, \cupleft,\, \capright, \, \capleft, \, X_{+}, \, X_{-},
\end{equation}
subject to relations that arise from Redemeister transformations and isotopies of tangle diagrams.

\subsubsection*{Quandles} A \emph{quandle} is a set $Q$ equipped with a binary operation $\triangleright: Q\times Q \longrightarrow Q$ satisfying:
\begin{itemize}
\item[(P1)] $x\triangleright x= x$, for all $x\in Q$.
\item[(P2)] For all $x,z\in Q$, there exists a unique $y\in Q$ such that $x\triangleright y=z$.
\item[(P3)] $ x\triangleright (y\triangleright z) = (x\triangleright y)\triangleright (x\triangleright z)$, for all $x,y,z\in Q$.
\end{itemize}

A \emph{quandle homomorphism} is a map $f:Q\rightarrow Q'$ preserving the quandle operation, i.e.\ $f(x\triangleright_Q y)=f(x)\triangleright_{Q'}f(y)$ for all $x,y \in Q$, so that quandles define a category, denoted by $\Qdl$. We note that uniqueness in (P2) leads to a second operation $\triangleleft$, defined as $z \triangleleft x = y$ if and only if $x \triangleright y =z$.

We note that a quandle can be constructed from any group $G$  by defining the binary operation
$$
a \triangleright b = aba^{-1}, \quad \text{for } a,b \in G,
$$
which endows $ G$ with the structure of a quandle, known as the \emph{conjugation quandle}, and denoted $\Conj(G)$.
Conversely, for any quandle $Q$, one defines the \emph{associated group} as
$$
\As(Q) = \left\langle a \in Q \,\middle|\, a \triangleright b = aba^{-1},\ \text{for all } a,b \in Q \right\rangle,
$$
which encodes the quandle operation in group-theoretic terms.

These constructions define an adjunction between the category of groups $ \mathbf{Grp}$ and the category of quandles $\Qdl$
\begin{equation} \label{eqn:adjunctiongrqdl}
\Hom_{\mathbf{Grp}}(\As(Q), G) \cong \Hom_{\Qdl}(Q, \Conj(G)),
\end{equation}
for any group $G$ and quandle $Q$ \cite{racks}. We are interested in the following quandles:
\begin{itemize}
    \item The \emph{knot quandle} $Q_K$ of any knot $K$, introduced by Joyce and Matveev \cite{Joycequandle,Matveev}. It is a complete knot invariant up to orientation, and it can be described geometrically or algebraically, in terms of a planar projection of the knot, providing a Wirtinger-type presentation of the quandle in terms of generators and relations, so that $\As(Q_K)\cong \pi_1(S^3-K)$.
    \item Given any set $S$ and the free group $\Free{S}$ that it generates, the \emph{free quandle} on $S$ is the set $\FQ_S \subset \Free{S}$  consisting of all conjugates in $\Free{S}$ of the elements of $S$, that is, 
    $$\FQ_S=\{ w s w^{-1} | w\in \Free{S}, s\in S\}.$$ The quandle operation is defined by $a \triangleright b= aba^{-1}$, so that the free quandle is a subquandle of $\Conj(\mathbf{F}_S)$.
    \item Given any $t \in \CC^* = \CC - \{0\}$, the \emph{$t$-Alexander quandle} is $\AQ{t} = (\CC, \triangleright_t)$ with the operations
    \begin{equation} \label{eqn:Alexquandlerel}
	a\triangleright_t b=tb+(1-t)a, \qquad b \triangleleft_t a = t^{-1}b + (1-t^{-1})a, \end{equation}
for $a, b \in \CC$. This Alexander quandle is subquandle of $\Conj(\agl)$, via the quandle morphism 
$$
        \AQ{t} \to \Conj(\agl), \qquad a \mapsto \begin{pmatrix}
            1 & 0 \\ a & t
        \end{pmatrix}.
    $$
\end{itemize}

\subsubsection*{The Category $\Span(\Vect_{\CC})$}

We briefly recall the main features of the category $\Span(\Vect_{\CC})$, where $\Vect_{\CC}$ denotes the category of finite-dimensional complex vector spaces:

\begin{itemize}
    \item \textbf{Objects:} same as in $\Vect_{\CC}$; that is, finite-dimensional complex vector spaces.
    
    \item \textbf{Morphisms:} For $X, Y \in \Vect_{\CC}$, a morphism from $X$ to $Y$ is represented by a \emph{span} $(Z,f,g)$ of linear maps,
    $$
    \xymatrix{
        & Z \ar[dl]_{f} \ar[dr]^{g} & \\
        X & & Y
    }
    $$
    i.e., by linear maps $f: Z \to X$, $g: Z \to Y$. Morphisms are equivalence classes of spans: $(Z, f, g) \sim (Z', f', g')$ if there exists an isomorphism $\varphi: Z \to Z'$ such that $f = f' \circ \varphi$ and $g = g' \circ \varphi$.
    
    \item \textbf{Composition:} Given morphisms $(Z_1, f_1, g_1): X \to Y$ and $(Z_2, f_2, g_2): Y \to Z$, their composition is defined by the pullback span
    $$
    \xymatrix{
        & & Z_1 \times_Y Z_2 \ar[dl]_{\pi_1} \ar[dr]^{\pi_2} & & \\
        & Z_1 \ar[dl]_{f_1} \ar[dr]^{g_1} & & Z_2 \ar[dl]_{f_2} \ar[dr]^{g_2} & \\
        X & & Y & & Z
    }
    $$
    where $Z_1 \times_Y Z_2 = \{(z_1, z_2) \in Z_1 \times Z_2 \mid g_1(z_1) = f_2(z_2)\}$, the fibre product in $\Vect_{\CC}$. 
\end{itemize}

\begin{rem}
There is a natural fully faithful embedding $\Vect_{\CC} \hookrightarrow \Span(\Vect_{\CC})$, where a linear map $f: X \to Y$ is sent to the span
$$
\xymatrix{
    & X \ar[ld]_{\id_X} \ar[rd]^{f} & \\
    X & & Y
}
$$
These will be referred to throughtout the text as \emph{basic spans}. This embedding respects composition: the composition of two basic spans corresponds to the span of the composed linear map, as the fiber product simplifies to
$$
\xymatrix{
    & & X = X \times_Y Y \ar[dl]_{\id_X} \ar[dr]^{f} & & \\
    & X \ar[dl]_{\id_X} \ar[dr]^{f} & & Y \ar[dl]_{\id_Y} \ar[dr]^{g} & \\
    X & & Y & & Z
}
$$
which corresponds to the basic span of $g \circ f$.
\end{rem}

\begin{rem}
If $X \stackrel{f}{\longleftarrow} Z \stackrel{g}{\longrightarrow} Y$ is a span where $f$ an isomorphism, then it is uniquely equivalent to the basic span with legs $\id_X$ and $g\circ f^{-1}$, that is,
$$
\xymatrix{
  & Z \ar[ld]_{f} \ar[rd]^{g} \ar@{--{>}}[dd]^{f} & \\
  X & & Y \\
  & X \ar[lu]^{\id_X} \ar[ru]_{\,g \circ f^{-1}} &
}
$$
the equivalence induced by the isomorphism $f: Z\to X$.
\end{rem}

Finally, we point out that the category $\Vect_\CC$ is a monoidal category with the direct sum $\oplus$ of vector spaces. This monoidal structure is inherited by $\Span(\Vect_\CC)$ through direct sum of both objects and spans.

\subsection*{The functors $\cA$ and $\cA_t$} We outline now the main ideas behind the construction of the functors $\cA$, $\cA_t$, and we choose generators for our computations, following \cite{AGL1Alexanderinvariants}. 

Given any tangle $T$, we can associate to it the span of $\aglC$-representation varieties
\begin{equation} \label{eqn:spanAGL1repvarieties}
\xymatrix{
& \R_{\aglC}(\Gamma_T) \ar[dl]_{i} \ar[dr]^{j} & \\
\R_{\aglC}(\Gamma_\epsilon) & & \R_{\aglC}(\Gamma_\varepsilon)
}
\end{equation}
where $\epsilon$ and $\varepsilon$ are the boundary points $\partial T=\overline{\epsilon} \sqcup \varepsilon$, and $\Gamma_T$, $\Gamma_\epsilon$ and $\Gamma_{\varepsilon}$ are fundamental groups of the complement of the tangle in $\RR^2\times [0,1]$ and its boundary points in $\mathbb{R}^2$. The maps $i: \R_{\aglC}(\Gamma_T) \to \R_{\aglC}(\Gamma_\epsilon)$ and $j: \R_{\aglC}(\Gamma_T) \to \R_{\aglC}(\Gamma_\varepsilon)$ are maps at the level of representation spaces which are naturally induced by the inclusion of the boundaries in $T$. Besides, each of these representation varieties can be regarded as singular vector bundle, as there always is a regular morphism between representation varieties
\begin{equation}\label{eq:Alexander-fibration-intro}
\pi: \R_{\agl}(\Gamma) \longrightarrow \R_{\CC^{\ast}}(\Gamma) = \CC^*,
\end{equation}
\emph{the Alexander fibration}, induced by the projection of $\pi: \aglC \rightarrow \CC^*$ onto its diagonal entry. The maps $i$ and $j$ are compatible with their fibered structures, and the Alexander polynomial arises as the generator of the vanishing ideal of the exceptional locus of this fibration.

Fix  any $t \in \R_{\CC^{\ast}}(\Gamma_K) = \CC^*$. Considering the fibre $\pi^{-1}_K(t)$ naturally leads to the functor $\cA_t$ via its connection to quandle theory. When $t\in \CC^{\ast}$ is fixed in \eqref{eqn:spanAGL1repvarieties}, the adjunction \eqref{eqn:adjunctiongrqdl} yields a span of quandle representation varieties
\begin{equation} \label{eqn:spanAQquandles}
\xymatrix{
& \R_{\AQ{t}}(Q_T) \ar[dl]_{i} \ar[dr]^{j} & \\
\R_{\AQ{t}}(Q_\epsilon)  & & \R_{\AQ{t}}(Q_\varepsilon)
}
\end{equation}
where $Q_T, Q_{\epsilon}$ and $Q_{\varepsilon}$ are the quandles associated to each of these spaces. In other words, the inclusion of the boundaries of $T$ produces a span of quandles, which we represent onto the Alexander quandle $\AQ{t}$ \cite{cattabriga}. In particular, when $K$ is a knot, $\pi_K^{-1}(t)$ can be identified with the quandle representation variety
$$
\R_{\AQ{t}}(Q_K) = \Hom_\Qdl(Q_K, \AQ{t}),
$$
of the knot quandle.

\subsection*{$\cA_t$ on the set of generators of $\cT$} \label{sec:generatorsdescription} We describe now explicitly the functor $\cA_t$ vía its action on the generators of $\cT$ listed in \eqref{eqn:generatorstanglecat}, using a natural presentation of the associated fundamental groups as follows.

\begin{itemize}
	\item Objects. Fix an object $\epsilon\in\mathcal T$ with length equal to $n:=\abs{\epsilon}$. The complement $\epsilon^{c}$ in $\mathbb{R}^2$ has fundamental group $\Gamma_\epsilon:=\pi_1(\epsilon^{c})\cong \Free{n}$. We choose orientations of the standard generators according to the sign of the corresponding point, so that from a top view the loops wind clockwise or counterclockwise (see Figure~\ref{fig:pi1generators}).
	
\begin{figure}[h]
    \centering
    \includegraphics[width=0.5\linewidth]{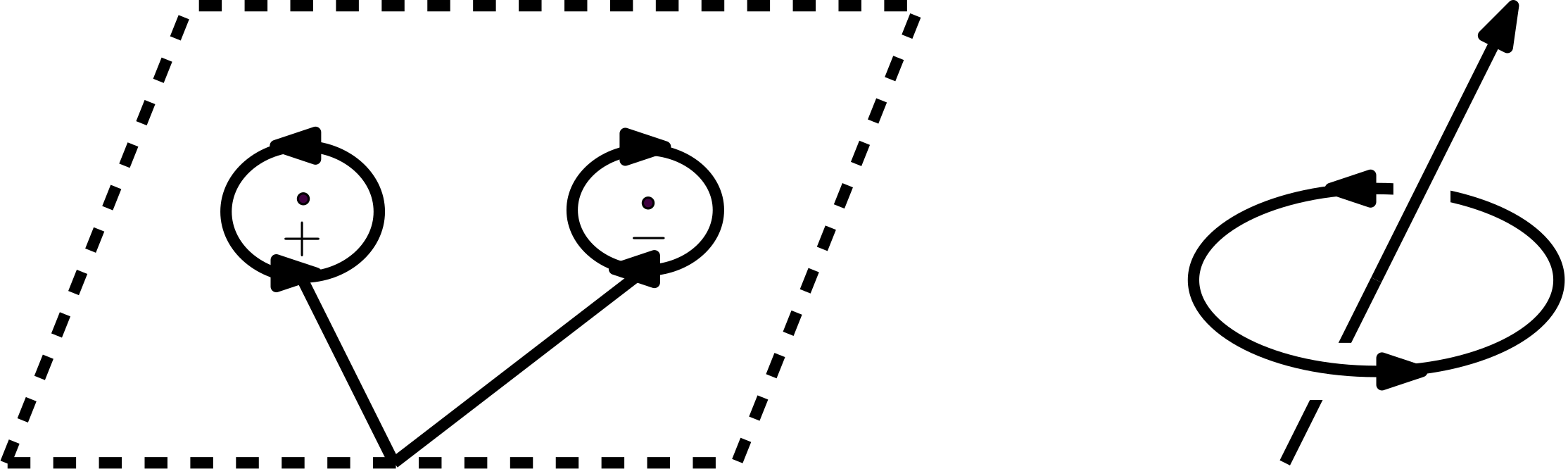}
    \caption{Generators of $\Gamma_\epsilon$ and $\Gamma_T$.}
    \label{fig:pi1generators}
\end{figure}

    The same criterion can be chosen for the quandle $Q_\epsilon$ when considering representatives of the set of homotopy classes of paths from a fixed basepoint to a neighborhood of $\epsilon$, so that we get a natural identification $\cA_t(\epsilon) = \R_{\AQ{t}}(\epsilon^c) = \CC^n$.
	
	\item Morphisms. Fix a tangle \(T:\epsilon\to\varepsilon\). The complement \(T^{c}\) has fundamental group \(\Gamma_T:=\pi_1(T^{c})\). We take as generators the meridians encircling each strand, oriented by the right--hand rule (see Figure~\ref{fig:pi1generators}). Again, this also defines generators of the quandle $Q_T$, leading to a description of $\R_{\AQ{t}}(T)$.
\end{itemize}

\subsubsection*{Generators}
With these choices on the fundamental group and the associated quandle for the six generators of $\cT$, we obtain the oriented generators depicted in Figure \ref{fig:pi1TangleGens}.
\begin{figure}[h]
    \centering
    \includegraphics[width=0.9\linewidth]{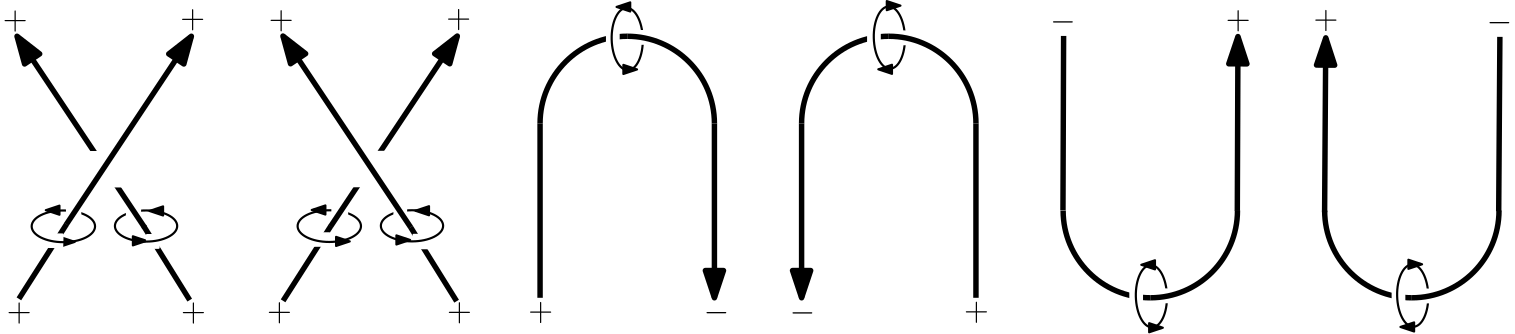}
    \caption{Orientations for the generators of the tangle category}
    \label{fig:pi1TangleGens}
\end{figure}
 It is proved in \cite{AGL1Alexanderinvariants} that the functor $\cA_t$ on these tangles produces the following spans of complex vector spaces:
\begin{itemize}
    \item Braids $X_+,X_-$: $\cA_t(X_+)$ and $\cA_t(X_-)$ are respectively given by the spans
    $$
 \xymatrix{
& \CC^2 \ar[ld]_{\id} \ar[rd]^{f_+}& \\
\CC^2 & & \CC^2
} \quad  \xymatrix{
& \CC^2 \ar[ld]_{\id} \ar[rd]^{f_-}& \\
\CC^2 & & \CC^2
}
$$
where $f_+,f_-:\CC^2\longrightarrow \CC^2$ are defined as:
\begin{align}
f_+(a_1, a_2) & =((1-t)a_1 + ta_2, a_1) \label{eqn:X+linearmap}\\
f_-(a_1, a_2) &  =(a_2, t^{-1}a_1+(1-t^{-1})a_2) \label{eqn:X-linearmap}
\end{align}

\item Evaluations and coevaluations: it is shown that $\cA_t(\capleft) = \cA_t(\capright)$ and also $\cA_t(\cupleft)=\cA_t(\cupright)$ are respectively given by the spans
$$
\xymatrix{
& \CC \ar[ld]_{\Delta} \ar[rd]& \\
\CC^2 & & 0
}\quad
\xymatrix{
& \CC \ar[rd]^{\Delta} \ar[ld]& \\
0 & & \CC^2
}
$$
where $\Delta:\CC \longrightarrow \CC^2$ is the diagonal map $\Delta(a)=(a,a)$.

We note that this last case naturally generalizes to the span attached to the evaluation and coevaluation maps in $\cT$. That is, given any object $\epsilon=(\epsilon_1,\ldots,\epsilon_n) \in \cT$ of length $n$, there is an evaluation tangle $\evalmap :\epsilon \, \otimes \epsilon^{\ast} \rightarrow \emptyset$, where $\epsilon^{\ast}=(-\epsilon_n,\ldots,-\epsilon_1)$ is the dual object. A similar construction produces the coevaluation tangle $\coevmap:\emptyset  \rightarrow\epsilon \, \otimes \epsilon^{\ast} $ The spans $\cA_t(\evalmap_{\epsilon})$  and $\cA_t(\coevmap_{\epsilon})$ are respectively equal to
\begin{equation} \label{eqn:evandcoevspans}
\xymatrix{
& \CC^n \ar[ld]_{(\id,c)} \ar[rd]& \\
\CC^{2n} & & 0
} \quad \xymatrix{
& \CC^n \ar[ld] \ar[rd]^{(\id,c)}& \\
0 & & \CC^{2n}
}
\end{equation}
where $c(z_1,\ldots,z_n)=(z_n,\ldots,z_1)$. 

The transpose $T^{\ast}$ of a tangle $T$, seen as a morphism between $\epsilon$ and $\varepsilon$, is the tangle $T^{\ast}=(\coevmap_{\epsilon}\otimes \id_{\varepsilon}) \circ (\id_{\epsilon}\otimes \, T\otimes \id_{\varepsilon}) \circ (\id_{\epsilon}\otimes \evalmap_{\varepsilon})$, where $\id_{\epsilon}=(\id)^{\otimes \abs{\epsilon}}$ is properly oriented. If $\cA_t(T)$ is the span $\CC^r \overset{f}{\longleftarrow} \CC^n \overset{g}{\longrightarrow} \CC^s$, then Equation  \eqref{eqn:evandcoevspans} implies that $\cA_t(T^{\ast})$ is the following span
\begin{equation} \label{eqn:transposespan}
\xymatrix{
& \CC^n \ar[ld]_{c \, \circ \, g} \ar[rd]^{c \, \circ \, f}& \\
\CC^s & & \CC^{r}
}
\end{equation}
\end{itemize}

\subsection{Quantum integers} Finally, we introduce notation and basic facts regarding the quantum integers $[n]_{t}$ that will be used throughout the text. Quantum integers are Laurent polynomials in $\ZZ[t,t^{-1}]$ that are defined as
$$
[n]_{t}:=\frac{1-t^n}{1-t}=t^{n-1}+\ldots+1,
$$
where $n\in \ZZ$ and we set $[0]_t=0$. We will make use of the following properties, which will be mostly used evaluated at $-t$:
\begin{itemize}
    \item [P1:]  \emph{(Recursiveness)} \begin{equation} \label{property:quantumint1} t[n]_t+1=[n+1]_{t}, \quad [1]_t=1.\end{equation}
    \item[P2:] \emph{(Negative integers)} \begin{equation} \label{property:quantumint2} [-n]_t=-\frac{1}{t}-\frac{1}{t^2}-\ldots-\frac{1}{t^n}=-t^{-n}[n]_t\end{equation}
    \item[P3:] \emph{(Evaluation at $t^{-1}$)} \begin{equation} \label{property:quantumint3} [n]_{t^{-1}}=t^{1-n}[n]_t=-t[-n]_t \end{equation}
\end{itemize}

ff\section{2-tangles} \label{sec:2-tangles}
We discuss in this section the effect of the functor $\mathcal{A}_t$ on 2-tangles, starting with the generators $X_+$ and $X_-$ of the tangle category and the braidings they generate. The analysis of these elementary tangles will allow us to study rational and pretzel tangles and links in Sections \ref{sec:2-bridgeknots} and \ref{sec:pretzelknots}.

\subsection*{Rotation and addition}
Applying a counterclockwise rotation by an angle of $\pi/2$ to  $X_+$ and $X_-$, introduced in Section \ref{sec:generatorsdescription}, yields eight elementary tangles
$$
\crossingXplus, \crossingXplusone, \crossingXplustwo, \crossingXplusthree, \crossingXminus, \crossingXminusone, \crossingXminustwo, \crossingXminusthree
$$
which differ from $X_+$ and $X_-$ only in their orientations. Together with the remaining generators $\cupright, \cupleft, \capright, \capleft$, this collection is sometimes considered as a generating set for the category of tangle diagrams;  see \cite{Kasselquantumgroups}.

The  examples above are 2-tangles: tangles with two interlacing strands and no loops, so that $\abs{b(T)}=\abs{s(T)}=2$. We index the endpoints by NE, NW, SE and SW according to their position in a planar diagram. The category of 2-tangles $\cT_2$ is special because it can be endowed with additional operations:
\begin{itemize}
    \item The \emph{sum} $T+T'$ of two tangles, which is the new tangle that is diagrammatically formed by attaching the NW and SW strands of $T$ to the NE and SE strands of $T'$, see Figure \ref{fig:tangleops}. It is defined when $SW(T)=-SE(T')$ and $NW(T)=-NE(T')$.
    \item The \emph{rotation} $T^r$ of a tangle $T$, which is the tangle that arises by performing a counter-clockwise rotation of angle $\pi/2$ on the plane.
    \item The \emph{mirror} image $\overline{T}$, obtained by switching all the crossings in $T$.
\end{itemize}

\begin{figure}[ht]
    \centering
    \begin{minipage}{0.45\textwidth}
        \centering
        \includegraphics[width=\textwidth]{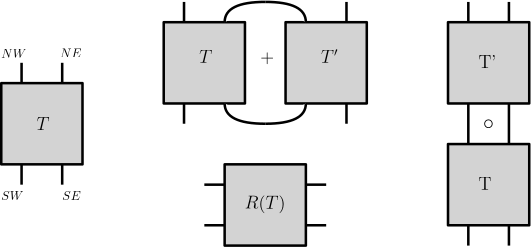}
        \caption{Operations for 2-tangles}
        \label{fig:tangleops}
    \end{minipage}
    \begin{minipage}{0.45\textwidth}
        \centering
        \includegraphics[width=0.7\textwidth]{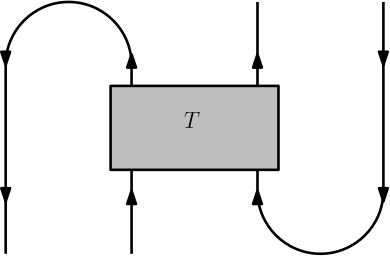}
        \caption{The tangle $R(T)$}
        \label{fig:rotationtangle}
    \end{minipage}
\end{figure}

In terms of a tangle decomposition, we can also define the rotation of a 2-tangle $T$ as follows.
\begin{defn} \label{defn:rotationtangle}
Given any 2-tangle $T$, we define the \emph{rotation} map $R:\mathcal{T}_2\longrightarrow\mathcal{T}_2$ as
$$
R(T)= (\; \cap \otimes \vert \otimes \vert \;) \circ (\; \vert \otimes  T \otimes \vert \;) \circ (\; \vert \otimes \vert \otimes \cup ),
$$
where $\vert$ stands for $\uparrow, \downarrow $ and $\cup, \cap$ also represent the tangles $\cupleft, \capleft$, depending on the orientation of $T$, so that $R(T)$ is well defined (see Figure \ref{fig:rotationtangle})
\end{defn}

The sum of two tangles $T,T'$ can be defined in terms of the rotation map as follows.
\begin{defn}
Let $T,T'$ be 2-tangles. The \emph{sum} $T+T'$ is defined as
$$
T+T'=R^{-1}(R(T')\circ R(T))
$$
whenever the composition $R(T)\circ R(T')$ is well defined.
\end{defn}

The following lemma, which we state without proof, summarizes the effect of rotating a 2-tangle $T$ at the level of spans. The result can be obtained by directly applying the functor $\cA_t$ to the tangle $R(T)$:
\begin{lem} \label{lem:rotation2tangle}
Let $T$ be a 2-tangle. Assume that $\cA_t(T)$ is the span
\[
\xymatrix{
& \CC^2 \ar[ld]_f \ar[rd]^g& \\
\CC^2 & & \CC^2
}
\]
where $f=(f_1,f_2)$, $g=(g_1,g_2)$. Then $\cA_t(R(T))$ is the span
\begin{equation} \label{eqn:rotationofspan}
\xymatrix{
& \CC^2 \ar[ld]_{(g_1,f_1)}\ar[rd]^{(g_2,f_2)}& \\
\CC^2 & & \CC^2
}
\end{equation}
\end{lem}

\begin{rem}
By extension, we will refer to the span \eqref{eqn:rotationofspan} as the rotation of the span $\cA_t(T)$.
\end{rem}
\begin{rem}
By the universal property of the product, the span \(\cA_t(T)\) determines a unique linear map
$$
\cA_t(T)\colon \CC^{2}\longrightarrow \CC^{4},\qquad 
\cA_t(T)=(f_{1},f_{2},g_{1},g_{2}).
$$
If \(r\colon \CC^{4}\to \CC^{4}\) denotes the permutation given by 
$
r(z_{1},z_{2},z_{3},z_{4})=(z_{3},z_{1},z_{4},z_{2}),
$ then $\cA_t(R(T))$ is determined by the linear map $ r\circ \cA_t(T) \colon \CC^2 \longrightarrow \CC^4$.
\end{rem}

The rotation $R$ has order four and restricts to an involution on unoriented rational tangles \cite{kauffmanrationaltangles}. Moreover, for $2$-tangles we have $R^{2}(T)=T^{\ast}$, so the span $\mathcal{A}_t(R^{2}(T))$ can be described via Lemma~\ref{lem:rotation2tangle} or Equation~\eqref{eqn:transposespan}. Analyzing the action of \(R\) on spans through the functor \(\mathcal{A}_t\) we obtain the following proposition.

\begin{prop} \label{prop:rotationspans}
\begin{align*}
\cA_t(R(X_+)) & =\cA_t(X_-), &  \cA_t(R^2(X_+)), & =\cA_{t^{-1}}(X_+), &  A_t(R^3(X_+)) & =\cA_{t^{-1}}(X_{-}), \\
\cA_t(R(X_-)) & =\cA_{t^{-1}}(X_+),  &   \cA_t(R^2(X_-)) & =\cA_{t^{-1}}(X_-), &  A_t(R^3(X_-))& =\cA_{t^{-1}}(X_{+}).
\end{align*}
\end{prop}
\begin{proof}
If we follow the conventions already set in Section \ref{sec:agl1repsandAlexanderinvariants}, let us recall that the basic spans $\cA_t(X_+)$ and $\cA_t(X_{-})$ are determined by the linear maps $f_+$, $f_-: \CC^2 \mapsto \CC^2$, where we recall that
\begin{equation} \label{eqn:f+f-functionequation}
f_+(a_1,a_2)=((1-t)a_1+ta_2,a_1), \quad f_{-}(a_1,a_2)=(a_2,t^{-1}a_1+(1-t^{-1})a_2)
\end{equation}
Although the maps $f_+,f_-$ depend on $t\in\CC^{\ast}$, we will emphasize this dependence only when necessary. Applying Lemma \ref{lem:rotation2tangle}, we obtain that $\cA_t(R(X_+))$ and $\cA_t(R(X_{-}))$ are the spans
\[
\xymatrix{
& \CC^2 \ar[ld]_{f_+} \ar[rd]^{\id}& \\
\CC^2 & & \CC^2
}\quad  \xymatrix{
& \CC^2 \ar[ld]_{c} \ar[rd]^{c\; \circ f_{-}}& \\
\CC^2 & & \CC^2
} 
\]
where $c$ is the permutation $c(a_1,a_2)=(a_2,a_1)$. Since $c \circ f_- \circ c \, (a_1,a_2)= ((1-t^{-1})a_1+t^{-1}a_2,a_1)$ and $f_+^{-1}=f_-$, we obtain that $\cA_t(R(X_+))=\cA_t(X_-)$ and $\cA_t(R^2(X_+))=\cA_t(R(X_-))=\cA_{t^{-1}}(X_+)$. Finally, $\cA_t(R^3(X_+))$ is given by the span
\[
 \xymatrix{
& \CC^2 \ar[ld]_{c} \ar[rd]^{c\; \circ f_{+}}& \\
\CC^2 & & \CC^2
} 
\]
which is equivalent to $A_{t^{-1}}(X_{-})$. Spans for $X_-$ follow from the equality $A_t(R(X_+))=A_t(X_-)$.
\end{proof}

We may observe that each leg of the spans that appear in Proposition \ref{prop:rotationspans} is an isomorphism. Therefore, we can restate Proposition \ref{prop:rotationspans}  in terms of basic spans, providing an second description of the effect of rotating a span with respect to the functor $\cA_t$. To ease notation, given linear maps $f,g:\CC^2\longrightarrow \CC^2$, let $(f,g)$ denote the span $\CC^2 \overset{f}{\longleftarrow} \CC^2 \overset{g}{\longrightarrow} \CC^2$. Besides, let $C(f)$ be the linear map $c\circ f \circ c$ and let $I(f)=f^{-1}$.
\begin{prop} \label{prop:diagramrotatedspans}
Consider the tangle generators $X_+=\crossingXplus$ and $X_-=\crossingXminus$, together with their rotations. The action of the functor $\cA_t$ on each of these generators is summarized in the following commutative diagram of basic spans:
\begin{equation} \label{eqn:diagramrotatedspans}
\xymatrix{
  \vcenter{\hbox{%
    $\begin{array}{c}
      \cA_t\!\left(\crossingXplus\right)=\cA_t\!\left(\crossingXminusthree\right)\\
      (\id, f_+(t))
    \end{array}$}} \ar[r]^I \ar[d]^C &
  \vcenter{\hbox{%
    $\begin{array}{c}
      \cA_t\!\left(\crossingXplusone\right)=\cA_t\!\left(\crossingXminus\right)\\
      (\id, f_-(t))
    \end{array}$}} \ar[d]^C
  \\
  \vcenter{\hbox{%
    $\begin{array}{c}
      \cA_t\!\left(\crossingXplusone\right)=\cA_t\!\left(\crossingXminustwo\right)\\
      (\id, f_-(t^{-1}))
    \end{array}$}} \ar[r]^I &
  \vcenter{\hbox{%
    $\begin{array}{c}
      \cA_t\!\left(\crossingXplustwo\right)=\cA_t\!\left(\crossingXminusone\right)\\
      (\id, f_+(t^{-1}))
    \end{array}$}}
}
\end{equation}

\end{prop}
\begin{proof}
The direct computations for the basic spans equivalent to those in Proposition~\ref{prop:rotationspans} are collected in Table~\ref{tab:basic_spans}.
\begin{table}[h!]
    \centering
    \renewcommand{\arraystretch}{1.4}
    \begin{tabular}{|c|c||c|c|}
        \hline
        \( \crossingXplus \) & Basic span linear map & \( \crossingXminus \) & Basic span linear map \\
        \hline
        \( \crossingXplus \) & \( f_+(t) \) & \( \crossingXminus \) & \( f_-(t) \) \\
        \( \crossingXplusone \) & \( f_-(t) \) & \( \crossingXminusone \) & \( f_+(t^{-1}) \) \\
        \( \crossingXplustwo \) & \( f_+(t^{-1}) \) & \( \crossingXminustwo \) & \( f_-(t^{-1}) \) \\
        \( \crossingXplusthree \) & \( f_-(t^{-1}) \) & \( \crossingXminusthree \) & \( f_+(t) \) \\
        \hline
    \end{tabular}
    \caption{Basic spans of \( X_+ \) y \( X_- \) and their rotations}
    \label{tab:basic_spans}
\end{table}
The algebraic properties of the maps $f_+$ and $f_-$, namely
$$
f_+=f_-^{-1}, \quad c\circ f_{\pm}(t)\circ c=f_{\pm}(t^{-1}), 
$$
complete the commutative diagram \eqref{eqn:diagramrotatedspans}.
\end{proof}
\begin{rem} \label{rem:rotatedspansdiagram}
Although several of the following properties follow directly from the definition of $\cA_t$, Proposition~\ref{prop:diagramrotatedspans} allows us to state the following consequences:
\begin{itemize}
\item Changing the orientation of the undercrossing strand of any tangle diagram generator has no effect on $\cA_t$.
\item $I^2=C^2=\id$, and $IC=CI$, so there is an action of the Klein four-group $F$ on the set of spans of the tangle diagram generators, with a single orbit of spans corresponding to the rotations of $\crossingXplus$ and $\crossingXminus$.
\item For any of these spans, since $T^*=R^2(T)$, duality corresponds to inverting $t\in \CC^{\ast}$, as  $\cA_t(T^*)=\cA_{t^{-1}}(T)$.
\item The linear maps of the basic spans $\cA_t(R^{4-i}(\crossingXminus))$ and $\cA_t(R^i(\crossingXplus))$ are inverses of each other. Geometrically $R^{4-i}(\crossingXminus)$ is the mirror image of $\cA_t(R^i(\crossingXplus))$, if crossings and orientations are both inverted.
\end{itemize}
\end{rem}

Extending the usual notion for knots, we define the \emph{reverse} of a tangle as the new tangle that originates when the orientation of every strand is reversed. The following corollary, expected from the nature of the Alexander module, also stems from Proposition \ref{prop:diagramrotatedspans}:

\begin{cor} \label{cor:spanreverse}
Let $T$ be a tangle and let $T'$ be its reverse. Then $\cA_t(T')=\cA_{t^{-1}}(T)$ 
\end{cor}

\begin{proof}
Given a tangle $T$, we may decompose it into a product and composition of elementary tangles and their rotations. In terms of these generators, a change of orientation corresponds to a rotation of angle $\pi$ of each of the generators $X_+,X_{-}$ and a change of orientation of the evaluation and coevaluation tangles $\cupright, \cupleft, \capright, \capleft$. It is straightforward to verify that $\cA_t(\cupleft)=\cA_t(\cupright)$ and also  $\cA_t(\capleft)=\cA_t(\capright)$ are independent of $t\in \CC^{\ast}$. Proposition \ref{prop:rotationspans} shows that $\cA_t(R^2(X_{\pm}))=\cA_{t^{-1}}(X_{\pm})$, which functorially implies that $\cA_t(T')=\cA_{t^{-1}}(T)$.
\end{proof}

\begin{cor}
Let $K$ be a knot and let $K'$ denote its reverse. Then $\cA_t(K')=\cA_{t^{-1}}(K)$.
\end{cor}

\begin{rem} \label{rem:algebraicinvolution} Note that Corollary~\ref{cor:spanreverse} can also be obtained through an algebro-geometric approach by viewing $\cA_t(T)$ as the ``fibres'' of the singular vector bundle $\cA(T)$. In this framework, considering the involution $\tau:\CC^{\ast}\to\CC^{\ast}$ that takes $t$ to $t^{-1}$, one can show that there exists an isomorphism of coherent sheaves 
$$
  \tau^{\ast}(\mathcal{M}_T)\cong\mathcal{M}_{T'},
$$
where $\mathcal{M}_T$ denotes the sheaf of sections of the Alexander fibration associated with the singular vector bundle $\cA(T)$, whose global sections form the Alexander module.
\end{rem}

\subsection*{Braidings}

We analyze in this subsection the image under $\cA_t$ of iterated compositions of the elementary tangles $X_+, X_{-}$ and their rotations. These oriented twisted braidings, formed by two interlacing strands, are shown in Figure \ref{fig:twistedbraidings}, and they are the basic building blocks in which we may decompose any rational tangle, presented in its 2-bridge form, or any pretzel tangle.

Since orientations play an important role in the functor $\cA_t$, we distinguish two types of oriented twisted braidings: \emph{codirected} (which are oriented braids) and \emph{alternating} (strands have opposite orientation). We recall that $f_{+}(t)$ and $f_-(t)$, defined in \ref{eqn:f+f-functionequation}, are the linear maps of the basic spans $\cA_t(X_+)$ and $\cA_t(X_-)$, we emphasize here their dependence on $t\in \CC^{\ast}$.
\begin{figure}[h]
\includegraphics[width=8cm]{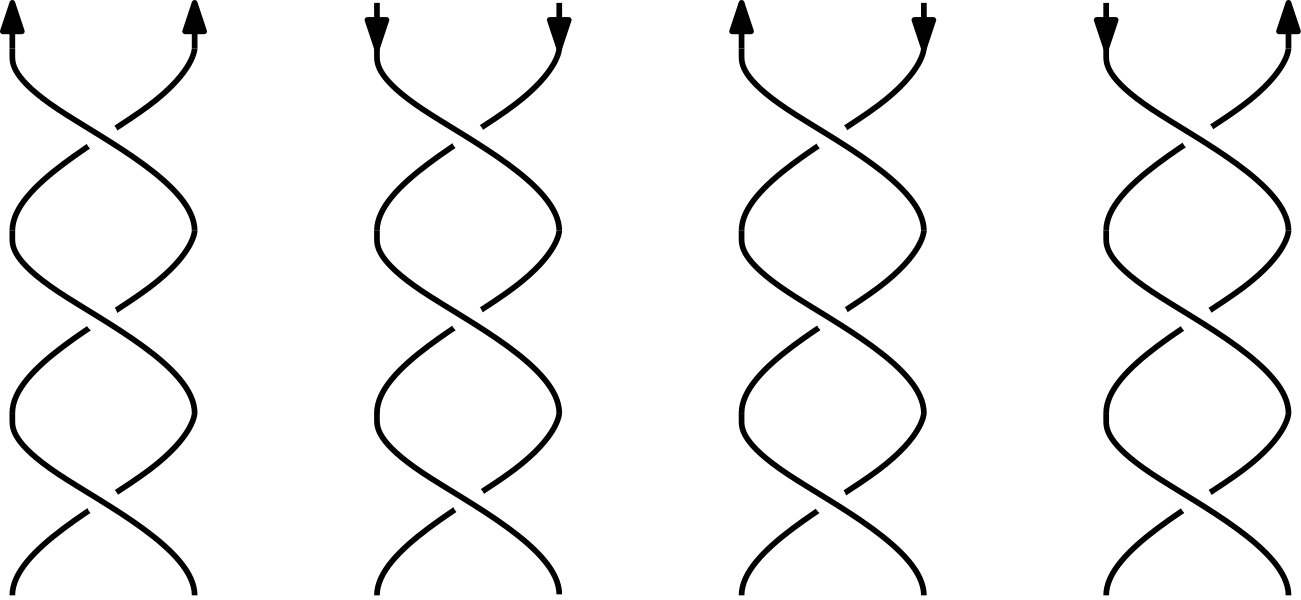}
\caption{Codirected and alternating braidings} \label{fig:twistedbraidings}
\end{figure}

\begin{prop}[Codirected braids] \label{prop:powersf+f-} $\cA_t(X_+^n), \cA_t(X_{-}^n)$ are the basic spans associated to the maps $f_+^n,f_{-}^n: \CC^2 \longrightarrow\CC^2$, defined for $n\geq 1$ as
\begin{align*}
f_{+}^n& =[n]_{-t} \,f_{+} + (1-[n]_{-t})\Id \\
f_{-}^n& =[n]_{-t^{-1} \,
}f_{-} + (1-[n]_{-t^{-1}})\Id 
\end{align*}
\end{prop}

\begin{proof}
The Cayley-Hamilton theorem applied to $f_{+}$ yields that $f_+^2=(1-t)f_+ +t\Id$. The result follows by induction, using the property of quantum integers \eqref{property:quantumint1}, that is, $[n+1]_{-t}=1-t[n]_{-t}$. The proof for $f_-^n$ is similar.
\end{proof}

We will use two equivalent expressions for the maps $f^n_+$ and $f^n_-$, specialized either at $t^{-1}$ or with $t$-factors omitted, as stated in the following corollary:
\begin{cor} \label{cor:X+-matrices}
With respect to the basis stablished in Section \ref{sec:agl1repsandAlexanderinvariants}, the linear operator of the basic span $\cA_{t}(X_{\pm})$ is represented by the matrices
$$
\cA_{t}(X_{\pm}^n)=\begin{pmatrix} 1-[\mp n]_{-t^{-1}}  & [\mp n]_{-t^{-1}} \\t^{-1}[\mp n]_{-t^{-1}} & 1-t^{-1}[\mp n]_{-t^{-1}} \end{pmatrix} =\begin{pmatrix} 1-[\mp n]_{-t^{-1}}  & [\mp n]_{-t^{-1}} \\ [\pm n]_{-t} & 1-[\pm n]_{-t} \end{pmatrix}
$$
\end{cor}

Alternating positive braidings are compositions of the tangles $R(X_+)$ and $R^3(X_+)$. Their image by $\cA_t$\ was computed in \ref{prop:rotationspans} and is  respectively given by the spans
\begin{equation}\label{eqn:RX+R3X+spans}
\xymatrix{
& \CC^2 \ar[ld]_{f_+(t)} \ar[rd]^{\id}& \\
\CC^2 & & \CC^2
}\qquad  \xymatrix{
& \CC^2 \ar[ld]_{c} \ar[rd]^{c\; \circ f_{+}(t)}& \\
\CC^2 & & \CC^2
} 
\end{equation}
\begin{prop}[Alternating braidings] \label{prop:twistedtanglesformulas}
Let $\widetilde{T}_+^n$ be the oriented tangle defined as:
\begin{itemize} 
\item $\widetilde{T}_+^{2k}=(R^3(X_+)\circ R(X_+))^k$.
\item $\widetilde{T}_+^{2k+1}=R(X_+)\circ (R^3(X_+)\circ R(X_+))^k$.
\end{itemize}
Then $\cA_t(\widetilde{T}_+^{2k})$ and $\cA_t(\widetilde{T}_+^{2k+1})$ are the basic spans associated to the maps $\tilde{f}^n_+: \CC^2 \longrightarrow \CC^2 $ defined as:
\begin{align} \label{eqn:twistedspans}
\tilde{f}_{+}^{2k}  & = \Id  - k(1-t^{-1})h \\
\tilde{f}_{+}^{2k+1} & = f_-(t) - k(1-t^{-1})h
\end{align}
where $h:\CC^2 \longrightarrow \CC^2$ is the nilpotent map $h(a_1,a_2)=(a_1-a_2,a_1-a_2)$.
\end{prop}
\begin{proof}
Proposition \ref{prop:rotationspans} states that $R(X_+)$ and $R^3(X_+)$ are the basic spans associated to the maps $f_-(t), f_-(t^{-1})$, respectively. Therefore, $\cA_t(\widetilde{T}_+^{2k})$ is the basic span associated to the map $(f_-(t^{-1}) \circ f_-(t))^k$. 

Since $f_-(t^{-1}) \circ f_-(t)$ is a unipotent map, which we may write as $\Id-(1-t^{-1})h$, it follows that $(f_-(t^{-1}) \circ f_-(t))^k$ equals $\Id-k(1-t^{-1})h$. Moreover, the map $f_-(t)\circ (f_-(t^{-1}) \circ f_-(t))^k$, which is the basic span of $\widetilde{T}_+^{2k+1}$, is equal to $f_-(t)-k(1-t^{-1})h$, since $f_-\circ h = h$.
\end{proof}

\begin{cor}
The basic spans of twisted alternating braidings of the form $(R(X_+)\circ R^3(X_+))^k$ and $R^3(X_+) \circ (R(X_+)\circ R^3(X_+))^k$, which are the reverses of $\widetilde{T}_{2k}^+$ and $\widetilde{T}_{2k+1}^+$, are also given by the maps $\tilde{f}^{2k}_+$ and $\tilde{f}^{2k+1}_+$ evaluated at $t^{-1}$, by Corollary  \ref{cor:spanreverse}.
\end{cor}

We may also define the negatively oriented version of $\widetilde{T}^{n}_+$, that is:
\begin{itemize} 
\item $\widetilde{T}_-^{2k}=(R^3(X_-)\circ R(X_-))^k$.
\item $\widetilde{T}_-^{2k+1}=R(X_-)\circ (R^3(X_-)\circ R(X_-))^k$.
\end{itemize}

The following proposition computes their associated basic spans: 
\begin{prop} \label{prop:twistedtanglesformulasII}
$\cA_t(\widetilde{T}_-^{2k})$ and $\cA_t(\widetilde{T}_-^{2k+1})$ are the basic spans associated to the maps $\tilde{f}^n_-: \CC^2 \longrightarrow \CC^2 $ defined as:
\begin{align} \label{eqn:twistedspans2}
\tilde{f}_{-}^{2k}  & = \Id  + k(1-t^{-1})h \\
\tilde{f}_{-}^{2k+1} & = f_+(t^{-1}) + k(1-t^{-1})h
\end{align}
\end{prop}
\begin{proof}
For $k\in \ZZ$, $\widetilde{T}^{2k}_- \circ \widetilde{T}^{2k}_+$ is the identity tangle, so the basic span of $\widetilde{T}^{2k}_-$ is the one given by the inverse of the map $\tilde{f}^{2k}_+$. Composing the latter with the basic span of $R(X_-)$, defined by the map $f_+(t^{-1})$ (see Proposition \ref{prop:diagramrotatedspans}), completes the formula in the odd case. 
\end{proof}
\section{2-bridge knots} \label{sec:2-bridgeknots}

Starting from the unoriented trivial $2$-tangle, one can construct new tangles---called \emph{rational tangles}---by composing and adding powers of the basic tangles $X_+$ and $X_-$ and their rotations. Equivalently, a rational tangle is any tangle produced by  successively applying a finite sequence of horizontal and vertical twists $a_1,\ldots,a_n$ to the four endpoints of the trivial tangle or its rotation ($[\infty]$ and $[0]$, respectively).

These tangles were classified by Conway \cite{conwayrational}, who established a bijection between them and the set of extended rational numbers $\mathbb{Q}\cup\{\infty\}$ via continued fraction expansions:
$$
\frac{p}{q} = [a_1,\ldots,a_n] =
a_1 + \cfrac{1}{a_{2}+\cfrac{1}{\ddots + \cfrac{1}{a_n}}}
\;\longmapsto\; T_{[a_1,\ldots,a_n]}.
$$
Conway observed that the arithmetic of these tangles, under addition, inversion, and composition, mirrors the arithmetic of rational numbers. We will revisit them in Section \ref{sec:functorA-1} from their coloring perspective \cite{kauffmanrationaltangles}.

Closing adjacent endpoints of any rational tangle (applying a \emph{numerator} or \emph{denominator closure}) produces a \emph{rational knot}, or a 2-bridge knot $b(p,q)$. These notions are equivalent, and the resulting links were first classified by Schubert using 2-fold branched covers. In what follows, we will use 2-bridge presentations of a rational knot $K$. Starting from a 3-strand braid model of a rational tangle $T=[a_1,\ldots,a_n]$, we tensor with a trivial strand and obtain the rational knot $b(p,q)$ by taking the appropiate closure determined by the parity of $n$, see Figure \ref{fig:closurerational}:

\begin{figure}[h]
\includegraphics[width=14cm]{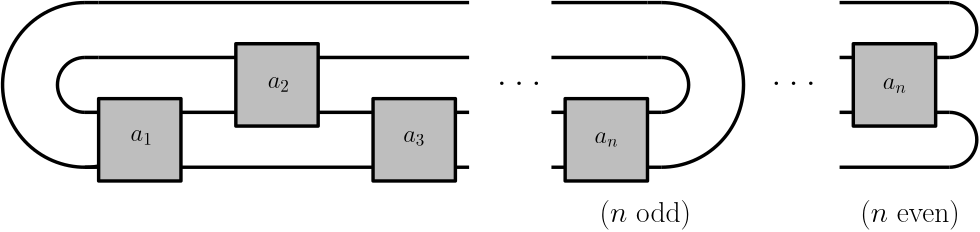}
\caption{Closures of a rational tangle in 3-strand form} \label{fig:closurerational}
\end{figure}

Propositions \ref{prop:powersf+f-}, \ref{prop:twistedtanglesformulas} and \ref{prop:twistedtanglesformulasII} will enable us to compute and study the Alexander invariants of any rational knot, via the functor $\cA_t$.  We begin by recalling the notion of the equalizer of a set of linear maps.
\begin{defn}
Let $V$ and $W$ be vector spaces and let $\mathcal{F}\subseteq \Hom(V,W)$ be a (finite) set of linear maps. The \emph{equalizer} of $\mathcal{F}$ is 
$$
\Eq(\mathcal{F})=\lbrace x\in V \mid f(x)=g(x) \; \text{ for all } f,g\in \mathcal{F} \rbrace.
$$
\end{defn}
When $\mathcal{F}=\lbrace f,g\rbrace$ consists of two maps, we write $\Eq(f,g)$; this is also known in the literature as the difference kernel. The next lemma details the effect of closing a rational tangle in the $\Span(\Vect_{\CC})$ category. 
\begin{lem}\label{lem:2bridgeequalizer}
Let $T$ be a rational tangle presented as a $3$-strand braid, and let $K$ be the associated $2$-bridge knot. Suppose $\cA_t(T\otimes \Id)$ is given by the span
$$
\xymatrix{
& \CC^{4} \ar[ld]_{f} \ar[rd]^{g} & \\
\CC^{4} & & \CC^{4}
}
$$
with linear maps $f,g:\CC^{4}\to\CC^{4}$, where we write $f=(f_1,f_2,f_3,f_4)$ and $g=(g_1,g_2,g_3,g_4)$. Then $\cA_t(K)$ equals
$$
\xymatrix{
& W \ar[ld] \ar[rd] & \\
0 & & 0
}
$$
where $W\subseteq \CC^{4}$ is the linear subspace
$$
W=
\begin{cases}
\Eq(f_1,f_4)\cap \Eq(f_2,f_3)\cap \Eq(g_1,g_4)\cap \Eq(g_2,g_3), & \text{if } n \text{ is odd},\\[2mm]
\Eq(f_1,f_4)\cap \Eq(f_2,f_3)\cap \Eq(g_1,g_2)\cap \Eq(g_3,g_4), & \text{if } n \text{ is even},
\end{cases}
$$
and $n$ denotes the length (hence parity) of the chosen $3$-strand braid presentation of $T$. 

We call $W$ the \emph{$2$-bridge equalizer} of $(f,g)$.
\end{lem}
\begin{proof} 
It is a direct consequence of composition in the span category, given by the pullback span of  $\cA_t(T\otimes \id)$ and the appropiate tensor product of the evaluation and coevaluation spans given in \eqref{eqn:evandcoevspans}, depending on the parity of $n$.
\end{proof}

\begin{rem}\label{rem:equalizermatrix}
Let $\cA_t(T)$ be a basic span $\CC^4 \overset{f}{\longleftarrow}\CC^4 \overset{g}{\longrightarrow} \CC^4$, i.e. $f=\id$. Then
$$
\Eq(f_1,f_2)\cap \Eq(f_3,f_4)
= \{(a,a,b,b)\in \CC^4 \mid a,b \in \CC\}.
$$
Consequently, $W$ can be described as the kernel of a linear map 
$$
g' \colon \CC^2 \longrightarrow \CC^2.
$$
Let $G$ denote the matrix of $g$ with respect to the canonical basis of $\CC^4$. Then, a matrix representation of $g'$ is obtained from $G$ by multiplication on the left and on the right with appropriate elementary matrices. Explicitly,
$$
W = \ker(g') = \ker(PG Q),
$$
where
\begin{equation} \label{eqn:closurematrices}
P_{odd}=\begin{pmatrix}
1 & 0 & 0 & -1 \\
0 & 1 & -1 & 0
\end{pmatrix},
\qquad
P_{even}=\begin{pmatrix}
1 & -1 & 0 & 0 \\
0 & 0 & 1 & -1
\end{pmatrix},
\qquad
Q=\begin{pmatrix}
1 & 0 \\
0 & 1 \\
0 & 1 \\
1 & 0
\end{pmatrix}.
\end{equation}
\end{rem}

Propositions \ref{prop:powersf+f-}, \ref{prop:twistedtanglesformulas}, \ref{prop:twistedtanglesformulasII} and Remark \ref{rem:equalizermatrix} provide the necessary ingredients to directly compute $\cA_t(K)$ from an oriented planar projection of $K$, as the next example shows.

\begin{ex} \label{ex:62knot}

\begin{figure}[h]
\includegraphics[width=6cm]{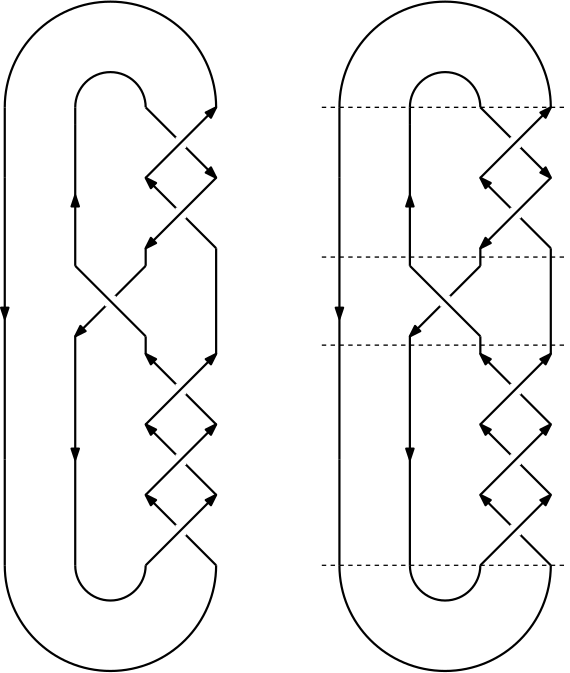}
\caption{The knot $6_2$, $[3,1,2]$} \label{fig:312}
\end{figure}
Consider the knot $K=6_2$, which is given as the closure of the rational tangle $[3,1,2]$ in Conway notation, represented as an oriented 2-bridge knot in Figure \ref{fig:312}. It may be decomposed as the composition of the following three tangles: $$T_1=\,\downarrow \otimes \, \downarrow \otimes \, X_+^3, \quad T_2= \,\downarrow \otimes \, R(X_+) \,\otimes \, \uparrow, \quad T_3=\,\downarrow \otimes \,\uparrow \otimes \, \widetilde{T}_-^2,\quad     $$
whose basic spans are respectively determined by the linear maps $\id \oplus \id \oplus f_+^3(t) $, $\id \oplus f_-(t) \oplus \id $ and $\id \oplus \id \oplus \tilde{f}_-^2(t)$. In terms of the basis given in Section \ref{sec:agl1repsandAlexanderinvariants}, these maps are respectively described by the matrices
$$
\begin{pmatrix} 1 & 0 & 0 & 0 \\ 0 & 1 & 0 & 0 \\0 & 0 & 1-t+t^2-t^3 & t-t^2+t^3 \\ 0 & 0 & 1-t+t^2 & t-t^2 \end{pmatrix}, \quad
\begin{pmatrix}
    1 & 0 & 0 & 0 \\ 0 & 0 & 1 & 0 \\ 0 & t^{-1} & 1-t^{-1} & 0 \\0 & 0 & 0 & 1
\end{pmatrix}, \quad 
\begin{pmatrix}
1 & 0 & 0 & 0 \\ 0 & 1 & 0 & 0 \\0 & 0 &2-t^{-1} & -1+t^{-1} \\0 & 0 & 1-t^{-1} & t^{-1}
\end{pmatrix},
$$
Thus the basic span of their composition is determined by the matrix
$$
G= \begin{pmatrix}
    1 & 0 & 0 & 0 \\ 0 & 0 & 1-t+t^2-t^3 & t-t^2+t^3 \\ 0 & -t^{-2}+2t^{-1} & t^{-2}-3t^{-1}+4-4t+4t^2-2t^3 & t^{-1}-3+4t-4t^{2}+2t^3 \\0 &-t^{-2}+t^{-1}& t^{-2}-2t^{-1}+3-3t+3t^2-t^3 & t^{-1}-2+3t-3t^2+t^3
\end{pmatrix}.
$$
By Remark \ref{rem:equalizermatrix}, $\cA_t(K)$ is the basic span $0\longleftarrow W\longrightarrow 0$, where $W$ is the kernel of $g'$, with associated matrix
$$
G'=\begin{pmatrix}
    -\Delta(t)& \Delta(t) \\ -\Delta(t) & \Delta(t)
\end{pmatrix}, \quad \Delta(t)=t^3-3t^2+3t-3+\frac{1}{t}.
$$
 The matrix $G'$ presents the Alexander module of $K$, and by Theorem \ref{thm-intro:TQFT-global}, yields the Alexander polynomial $\Delta(t)$ of $K$.
\end{ex}

Example \ref{ex:62knot} shows that $\cA_t$ computes the Alexander module of a knot $K$ from any planar diagram. If $K$ is a rational knot with associated fraction $p/q$, a continued-fraction expression of the form 
$$
\frac{p}{q}=[a_1,\ldots,a_n], \quad a_i\in \mathbb{Z}_{+}, \quad n \text{ odd},
$$
is said to be in \emph{canonical form}; this representation is unique \cite{kauffmanrationaltangles}. Other generalized Euclidean algorithms yield different, but again unique, continued-fraction expansions obtained by imposing certain restrictions on $a_i$. For our purposes, we recall the following expansion, which is particularly convenient for rational tangles \cite{Cromwell}:
\begin{prop}\label{prop:continuedfractioneven}
Let $p>q>0$, $\gcd(p,q)=1$, $q$ even. There exists a unique continued fraction expansion of the form
$$
\frac{p}{q}=[2k_1,2k_2,\ldots,2k_n], \quad k_i\in \ZZ_+.
$$
\end{prop}
\begin{rem} Since the rational knot $K=b(p,q)$ and its mirror $\overline{K}=b(p,-q)$ have the same Alexander module, we may assume that $q$ is even by choosing the appropiate mirror of $K$, due to Schubert's classification. Besides, we also note that when all coefficients in the continued fraction expansion are even, the parity of $n$ determines the number of components of the link: it is a knot when $n$ is even and a 2-component link when $n$ is odd (see Figure \ref{fig:closurerational}).
\end{rem}
Let $K$ be a rational knot presented as the numerator closure of a rational tangle $T_{p/q}$. Any continued-fraction expansion of $p/q$ gives rise to a planar 2-bridge diagram for $K$. If the  expansion has even length and all coefficients are even, then every braiding block in the diagram is alternating. We may orient $K$ so that, in the tangle corresponding to the terminal coefficient $2k_{1}$, the first generator is $R^3(X_\pm)$. With this choice, the first twist in each subsequent box associated to $2k_i$ alternates between $R^3(X_\pm)$ and $R(X_\pm)$ as $i$ increases. As is customary for planar diagrams of rational knots, when $2k_i>0$, (unoriented) twists are positive for even $i$ and negative for odd $i$; the signs are reversed when $2k_i<0$. These conventions lead to the following proposition:

\begin{prop}\label{prop:closedformulaAlex}
Let $K$ be a rational knot $b(p,q)$ and let us consider the even continued fraction expansion $\frac{p}{q}=[2k_1,\ldots,2k_{2n-1},2k_{2n}]$, where $p>q>0$, $\gcd(p,q)=1$, $q$ even. Then $\cA_t(K)$ is the 2-bridge equalizer of the linear map
$$
(1\oplus \tilde{f}^{2k_{n}}_{s_{n}}(t) \oplus 1)\circ (1 \oplus 1 \oplus \tilde{f}^{2k_{n-1}}_{-s_{n-1}}(t^{-1}))\circ \ldots (1 \oplus \tilde{f}^{2k_{n-1}}_{s_{2}}(t) \oplus 1) \circ (1 \oplus 1 \oplus \tilde{f}^{2k_{1}}_{-s_{1}}(t^{-1})),
$$
where $s_i=\sign (k_{i})$

\end{prop}

\begin{ex} \label{ex:3/11from2bridge}
We revisit Example \ref{ex:62knot} and apply Proposition \ref{prop:closedformulaAlex}. The rational knot $6_2$ is the 2-bridge knot $b(11,3)$, with mirror $b(11,8)$. If we apply Proposition \ref{prop:continuedfractioneven}, we get the following continued fraction expansion: 
$$
11/8=[2,-2,2,2].
$$
Therefore, $\cA_t(K)$ is the 2-bridge closure of the map 
$$
(1\oplus \tilde{f}^2_+(t) \oplus 1) \circ (1\oplus 1\oplus \tilde{f}^2_-(t^{-1}))\circ (1 \oplus \tilde{f}^2_-(t)\oplus 1) \circ (1 \oplus 1 \oplus \tilde{f}^2_-(t^{-1})),
$$
whose associated matrix is
$$
G=\begin{pmatrix}
    1 & 0 & 0 & 0 \\ 0 & t^{-2}-3t^{-1}+4-t & -2t^{-2}+6t^{-1}-7+4t-t^2 & t^{-2}-3t^{-1}+4-3t+t^2 \\
    0 & t^{-2}-4t^{-1}+5-2t  & -2t^{-2}+8t^{-1}-9+6t-2t^2 & t^{-2} -4t^{-1}+5-4t+2t^{2}\\ 0 & -t^{-1}+2-t & 2t^{-1}-3+2t-t^2 &  -t^{-1}+2-t+t^2
\end{pmatrix}.
$$
Taking the appropiate closure considering that $n$ is even, $A_t(K)$ is the span $0\longleftarrow Z\longrightarrow0$, where $Z=\ker G'$, 
$$
G'=\begin{pmatrix}
\Delta(t) & -\Delta(t) \\ -\Delta(t) & \Delta(t)
\end{pmatrix}, \qquad \Delta(t)=-t^{2}+3t -3+3t^{-1}-t^{-2}.
$$
matching Example \ref{ex:62knot}.
\end{ex}

 Proposition \ref{prop:closedformulaAlex} may be used to obtain a recursive formula for the Alexander polynomial of any rational knot from the description of $\cA_t(K)$. We consider the even continued-fraction expansion $p/q=[2k_1,\ldots,2k_{2n}]$ under the previous assumptions, and we use the standard basis for the 3-strand braid model of the rational tangle $T_{p/q}$, following the conventions set in Section \ref{sec:agl1repsandAlexanderinvariants}.  Let us write
$$
A_{2i}=\diag(1,A'_{2i},1), \qquad A_{2i-1}=\diag(1,1,A'_{2i-1}), \quad H=\begin{pmatrix}
    1 & -1 \\ 1 & -1
\end{pmatrix},
$$
where 
\begin{align*}
& A'_{2i-1}  =\begin{pmatrix}
    1+k_{2i-1}(1-t) & -k_{2i-1}(1-t) \\ k_{2i-1}(1-t) & 1-k_{2i-1}(1-t) \end{pmatrix}=\id +k'_{2i-1}H, \qquad  &k'_{2i-1} :=k_{2i-1}(1-t), \\
& A'_{2i}  =\begin{pmatrix}1-k_{2i}(1-t^{-1}) &  k_{2i}(1-t^{-1})\\ -k_{2i}(1-t^{-1}) & 1+k_{2i}(1-t^{-1})
\end{pmatrix}=\id+k'_{2i}H, \qquad & k'_{2i} :=-k_{2i}(1-t^{-1}),
\end{align*}

where we also recall that $h$ is the nilpotent map defined in Proposition \ref{prop:twistedtanglesformulas}. Let us also write $P,Q$ for the matrices from \eqref{eqn:closurematrices} that correspond to the closure of the rational tangle T. Proposition \ref{prop:closedformulaAlex} establishes that $\cA_t(K)$ is given by the span
$$
\xymatrix{
& \ker N \ar[ld] \ar[rd]& \\
0 & & 0
}
$$
where
\begin{equation}\label{eqn:PA_iQ}
N=PA_{2n}A_{2n-1}\ldots A_{2}A_{1}Q,
\end{equation}
We shall obtain an explicit description in terms of continuants of the Alexander polynomial of $T$. We recall their definition:

\begin{defn}
Let
$$
T_n = \begin{pmatrix}
a_{1} & u_{1} & 0 & \cdots & 0 \\
l_{2} & a_{2} & u_{2} & \ddots & \vdots \\
0 & l_{2} & a_{3} & \ddots & 0 \\
\vdots & \ddots & \ddots & \ddots & u_{n-1} \\
0 & \cdots & 0 & l_{n} & a_{n}
\end{pmatrix}
$$
be a tridiagonal matrix with complex coefficients. The \emph{continuant} associated with $T$, denoted
\[
K_{n}(a_{1},\dots,a_{n}; l_{2},\dots,l_{n}; u_{1},\dots,u_{n-1}),
\]
is the determinant of $T_n$, and can be recursively defined by $K_{0} = 1,\; K_{1}(a_{1}) = a_{1}$, and for $n \geq 2$,
\begin{equation}\label{eqn:recursioncontinuant}
K_{n}
= a_{n}\,K_{n-1} - l_{n}u_{n-1}\,K_{n-2}.
\end{equation}
\end{defn}
In this section, we will restrict to the case where $u_i=l_i=1$ for all $i$, first studied by Euler and classically connected with continued fractions;  we will write $K_n:=K_n(a_1,\ldots,a_n)$.

\begin{thm}
Let $K$ be a rational knot such that $p/q=[2k_1,\ldots,2k_{2n}]$. Then
\begin{equation} \label{eqn:ContinuantRationalAlexformula}
\Delta_K(t)=K_{2n}(k_1(1-t),-k_2(1-t^{-1}),\ldots,k_{2n-1}(1-t),-k_{2n}(1-t^{-1}))
\end{equation}
\end{thm}
\begin{proof}
Let us write $(a_i,b_i,c_i,d_i)=A_i A_{i-1}\ldots A_1(a_0,b_0,c_0,d_0)$. Also, we set
$$
x_i:=b_i-c_i, \qquad y_i:=c_i-d_i.
$$
Note that $a_i=a_0$ is constant for all $i$, so it suffices to obtain a recursive formula for $x_i,y_i,d_i$, as all the information is contained in the lower $3\times 3$-block of $A_i$. We note that, given the columns of $Q$, we may set the initial conditions $(b_0,c_0,d_0)=(0,0,1)$ and $(b_0,c_0,d_0)=(1,1,0)$, that is,
$$
(x_0^{(1)},y_0^{(1)},d_0^{(1)})=(0,-1,1), \quad (x_0^{(2)},y_0^{(2)},d_0^{(2)})=(0,1,0),
$$
and accordingly $(x_i^{(l)},y_i^{(l)},d_i^{(1)}),\; l=1,2,$ so that the superscript denotes the two possible initial conditions, ommited where needed to ease the notation.

Let us write $T_i=A_{2i}A_{2i-1}$. 
Given the definitions of $A_{2i-1}$ and $A_{2i}$, it can be checked that
$$
A_{2i-1}(x,y,d)=(x-k'_{2i-1}y,y,k'_{2i-1}y+d), \quad A_{2i}(x,y,d)=(x,k'_{2i}x+y,d),
$$
so that 
$$
T_i(x,y,d)=(x-k'_{2i-1}y, k'_{2i}x+(1-k'_{2i-1}k'_{2i})y, k'_{2i-1}y+d) .
$$
Note that $x_i+d_i$ are invariant by $T_i$. Let us show by induction that
\begin{equation}\label{eqn:xmymrecursion}
x_{2m}=(-1)^mK_{2m-1}(k_1',\ldots,k'_{2m-1}), \quad y_{2m}=(-1)^{m}K_{2m}(k_1',\ldots,k'_{2m}).
\end{equation}
It is clear that $x_0=0=K_{-1}$, $y_0=\mp1=\mp K_0$, assume that \eqref{eqn:xmymrecursion} holds for $2m\in \mathbb{N}$. First,
\begin{align*}
y_{2m+2} & = k'_{2m+2}x_{2m}+(1-k'_{2m+1}k'_{2m+2})y_{2m} = k'_{2m+2}(-1)^m K_{2m-1}+(1-k'_{2m+1}k'_{2m+2})(-1)^mK_{2m} \\& = (-1)^{m+1}\left( (k'_{2m+1}k'_{2m+2}-1)K_{2m}-k_{2m+2}'K_{2m-1} \right)  = (-1)^{m+1}K_{2m+2},
\end{align*}
given that two-step continuant identity $K_{2m+2}=(k'_{2m+2}k'_{2m+1}-1)K_{2m}-k'_{2m+2}K_{2m-1}$ from \eqref{eqn:recursioncontinuant} . Similarly,
\begin{align*}
x_{2m+2} & = x_{2m}-k'_{2m+1}y_{2m}= (-1)^mK_{2m-1}-k'_{2m+1}(-1)^mK_{2m} \\ & = (-1)^{m+1}\left(k_{2m+1}'K_{2m}-K_{2m-1}  \right) = (-1)^{m+1}K_{2m+1}.
\end{align*}
Now, a direct consequence of \eqref{eqn:PA_iQ} is that
$$
N=\begin{pmatrix} 
    a^{(1)}_{2n}-b_{2n}^{(1)} & 
    a^{(2)}_{2n}-b_{2n}^{(2)} \\y_{2n}^{(1)} & y_{2n}^{(2)} \end{pmatrix}.
$$
We note that $a_i$ and $x_i+d_i$ are invariant by $T_i$. Hence
\begin{align*}
a_{2n}^{(1)}-b_{2n}^{(1)}& =1-b_{2n}^{(1)}=-c_{2n}^{(1)}+d_{2n}^{(1)}=-y_{2n}^{(1)},  \\
a_{2n}^{(2)}-b_{2n}^{(2)}& =-b_{2n}^{(2)}=-c_{2n}^{(2)}+d_{2n}^{(2)}=-y_{2n}^{(2)}.
\end{align*}
Note that $y_i^{(2)}=-y_i^{(1)}$, as $y_0^{(2)}=1=-y_0^{(1)}$. Therefore
$$
N=(-1)^nK_{2n}(k'_1,k'_2,\ldots,k'_{2n-1},k'_{2n})\begin{pmatrix}
 1 & -1\\ -1 & 1 \end{pmatrix},
$$
Since $\cA_t(K)$ is the span determined by $\ker N$, the factor $K_{2n}(k_1',\ldots,k_{2n}')$ is (up to normalization) the Alexander polynomial of $K$.
\end{proof}
In terms of quantum integers, we may restate the theorem as follows.
\begin{cor}
Let $K$ be a rational knot such that $p/q=[2k_1,\ldots,2k_{2n}]$. Then
\begin{equation} \label{eqn:ContinuantRationalAlexformula}
\Delta_K(t)=K_{2n}([2]_{-t}k_1,-[2]_{-t^{-1}}k_2,\ldots,[2]_{-t}k_{2n-1},-[2]_{-t^{-1}}k_{2n}).
\end{equation}
\end{cor}
\begin{ex}
We consider again $11/8=[2,-2,2,2]$. It is an immediate check that
$$
K_4(1-t,1-t^{-1},1-t,-1+t^{-1})= -t^2+3t-3+3t^{-1}-t^{-2}=\Delta_K(t),
$$
matching Example \ref{ex:3/11from2bridge}.
\end{ex}
\begin{rem}
We note that our formula can also be obtained using a Seifert surface for $K$ \cite{LyubichMurasugi}. Several alternative closed expressions for the Alexander polynomials of rational knots appear in the literature. These include formulas derived from number-theorical functions in $p$ and $q$ \cite{Fukuhara}, as well as recursive constructions obtained from the skein relation as in \cite{Cromwell}, which yield a similar formula for the Conway polynomial. Given the relation between continued fractions and continuants, the latter can be expressed as a continuant; see \cite{KoseleffPecker} for this reinterpretation and for its relationship with  Fibonacci polynomials.
\end{rem}

\section{Pretzel knots} \label{sec:pretzelknots}

Pretzel knots are a well-studied family of knots where Propositions \ref{prop:powersf+f-}, \ref{prop:twistedtanglesformulas}, \ref{prop:twistedtanglesformulasII} may be directly applied to determine their Alexander modules. A \emph{Pretzel link} $P(q_1,\ldots,q_n)$ is a link that is obtained by cyclically connecting $n$ twisted bands, where each one has $q_i$ half twists. A Pretzel link is a knot if and only if \cite{Kawauchibook}:
\begin{enumerate}
    \item One $q_i$ is even, $q_j$ is odd for $j\neq i$.
    \item All $q_i$ are odd, $n$ is odd.
\end{enumerate}
As usual, the sign of $q_i$ determines if the unoriented twist is either positive or negative. Choosing an orientation for $K$, the following proposition lists the possible decompositions of $K$ as a tangle in terms of codirected braids or alternating braidings.
\begin{figure}[h]
  \centering
  \includegraphics[width=0.31\textwidth]{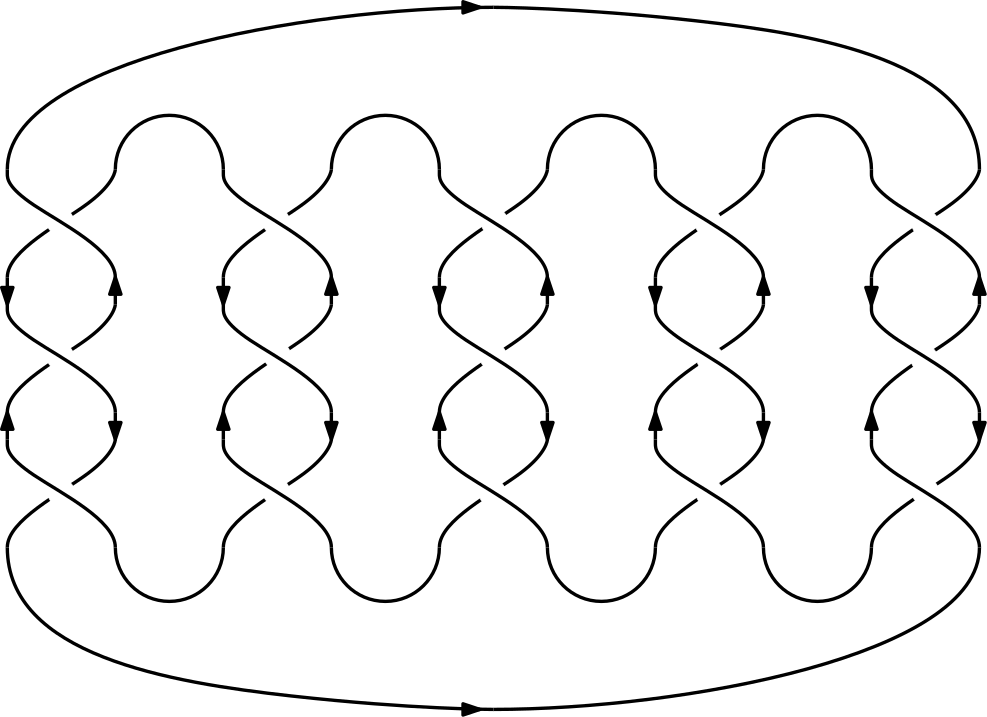}\hfill
  \includegraphics[width=0.31\textwidth]{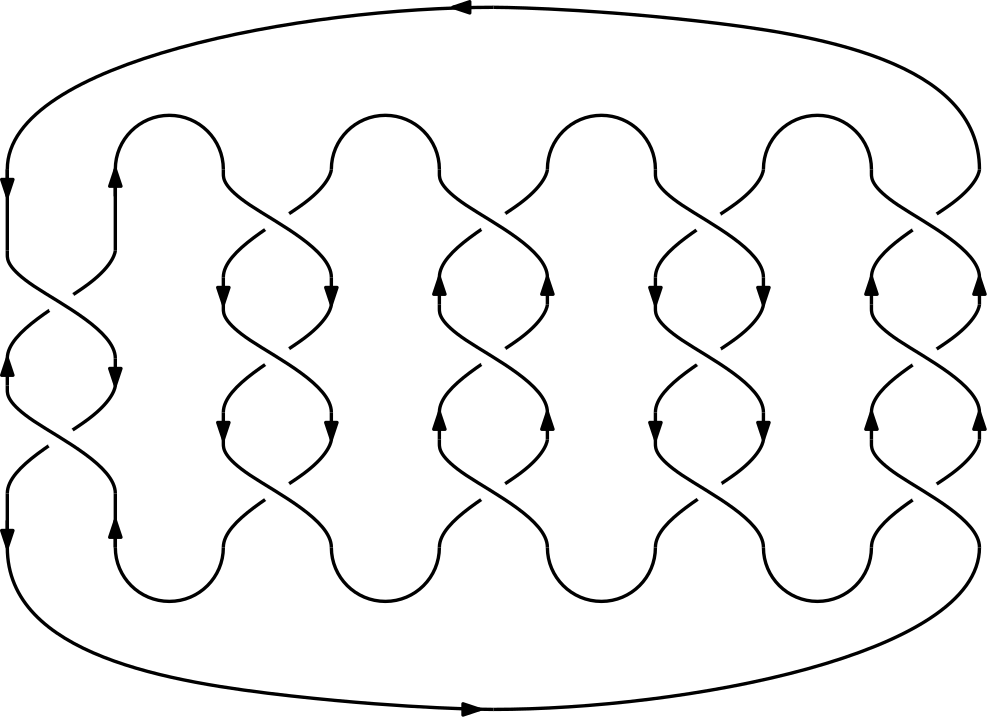}\hfill
  \includegraphics[width=0.31\textwidth]{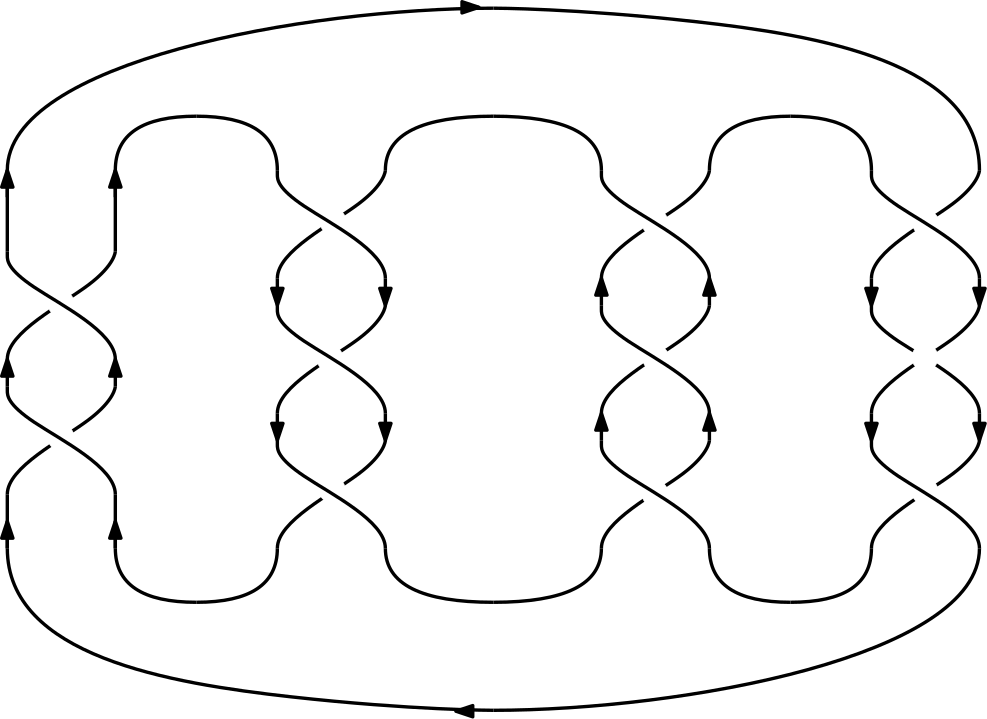}
  \caption{Oriented odd and even pretzel knots.}
  \label{fig:pretzelknot}
\end{figure}
\begin{prop}\label{prop:pretzeldecomp}
Let $K=P(q_1,\ldots,q_n)$ be an oriented pretzel knot.
\begin{itemize}
    \item If $q_i$ and $n$ are odd, then all twists are alternating braidings.   
    \item If $q_1$ is even and $n$ is even, all twists are codirected braids.
    \item If $q_1$ is even and $n$ is odd, all twists are codirected braids except for $q_1$, which is an alternating braiding.
\end{itemize}
\end{prop}
\begin{proof}
Any pretzel knot is the closure of a pretzel $2n$-tangle $T$, which is obtained as the tensor product of $n$ 2-tangles, each of which is a power of the unoriented generators $[X_+],[X_-]$. When oriented, each 2-tangle is either an alternating or codirected braiding.
Interpreting these 2-tangles as a morphisms in $\cT$, the parity of $q_i$ determines the sign of the endpoints and whether the strand is oriented upwards or downwards, according to the conventions in Section \ref{sec:agl1repsandAlexanderinvariants}. Choosing an orientation for $K$ and tracking it through the cyclic closure of $T$ yields the statement of the proposition.
\end{proof}

\begin{prop}\label{prop:pretzelspans}
Let $K = P(q_1,\ldots,q_n)$ be a pretzel knot. Then the associated span $\cA_t(K)$ is given by  
$$
0 \;\longleftarrow\; W \;\longrightarrow\; 0,
$$  
where $W$ is the kernel of the linear endomorphism  
$$
t_2 \circ D \circ t_1 : \CC^n \longrightarrow \CC^n,
$$
defined as follows: 
\begin{align*}
D   : \CC^{2n} \longrightarrow \CC^{2n}, &\quad 
    D   = \diag(f^{q_1}_{a_1}, \ldots, f^{q_n}_{a_n}), \\
t_1 : \CC^{n}  \longrightarrow \CC^{2n}, &\quad 
    t_ 1(z_1,\ldots,z_{n}) = (z_1,z_2,z_2,z_3,z_3,\ldots,z_{2n},z_1), \\
t_2 : \CC^{2n} \longrightarrow \CC^{n},  &\quad 
    t_2(z_1,\ldots,z_{2n}) = (-z_1+z_{2n},z_2-z_3,z_4-z_5,\ldots,z_{2n-2}-z_{2n-1}).
\end{align*}

Here, for each $i=1,\ldots,n$, the block $f^{q_i}_{a_i}$ (or its alternating counterpart $\tilde{f}^{q_i}_{a_i}$) is chosen according to the tangle decomposition specified in Proposition \ref{prop:pretzeldecomp}. 
\end{prop}

\begin{proof}
Let us decompose $T=C_2\circ T \circ C_1$, where $C_1,C_2$ are the cyclic pretzel closures of the $2n$-tangle $T$. It can be checked that $\cA_t(C_1)$ and $\cA_t(C_2)$ are respectively given by the spans
$$
\xymatrix{
	& \CC^{n} \ar[ld] \ar[rd]^{t_1} & \\
	0 & & \CC^{2n}
} \quad \xymatrix{
	& \CC^{n} \ar[ld]_{t_1} \ar[rd] & \\
	\CC^{2n} & & 0
} 
$$
Focusing now on $T$, the inclusion of $\partial T$ into $T$ is a homotopy equivalence, so $\cA_t(T)$ is equivalent to a basic span, which can be computed from the basic spans of each braiding $T_i$. Since $T=T_1\otimes T_2\otimes \ldots \otimes T_n$, we get that $\cA_t(T)=\cA_t(T_1)\oplus \cA_t(T_2)\oplus \ldots \oplus \cA_t(T_n)$, thus obtaining a basic span whose associated linear map is the block-diagonal map $D$. Finally, composing $\cA_t(C_2)\circ \cA_t(T)\circ \cA_t(C_1)$ yields the endospan of $0$ given by the equalizer of $D \circ t_1$ and $t_1$:
$$
\xymatrix{
	& \Eq(D\circ t_1,t_1) \ar[ld] \ar[rd]& \\
	0 & & 0
}
$$
which is isomorphic to $\ker (t_2\circ D\circ t_1)$, since the components of the map $t_2$ form a basis of the annihilator of the image of $t_1$.
\end{proof}

Proposition \ref{prop:pretzelspans} shows that the Alexander module of a pretzel knot can be expressed in terms of the matrix product $PDQ$, where $P$ and $Q$ are the $n \times 2n$ and $2n \times n$ matrices, respectively, given by  
\[
P = 
\begin{pmatrix} 
-1 & 0 & 0 & 0 & 0  &  \cdots & 0 & 0 & 1 \\
0 & 1 & -1 & 0 & 0 & \cdots & 0 & 0 & 0 \\
0 & 0 & 0 & 1 & -1 & \cdots & 0 & 0 & 0 \\
\vdots & & & & & & & & \vdots \\
0 & 0 & 0 & 0 & 0 & \cdots & 1 & -1 & 0
\end{pmatrix},  
\qquad 
Q = 
\begin{pmatrix} 
1 & 0 & 0 & \cdots & 0 \\ 
0 & 1 & 0 & \cdots & 0 \\ 
0 & 1 & 0 & \cdots & 0 \\ 
\vdots & & & & \vdots \\ 
0 & 0 & 0 & \cdots & 1 \\ 
0 & 0 & 0 & \cdots & 1 \\
1 & 0 & 0 & \cdots & 0
\end{pmatrix}.
\]  
Here $D$ denotes a block diagonal matrix whose diagonal blocks are $2\times 2$ matrices determined by the basic spans that corresponding to the alternating and codirected braidings of the pretzel tangle $T$. Note that the elementary row and column operations defined by the pretzel closure matrices $P$ and $Q$ only involve adjacent rows and columns if we regard indices modulo $n$, so that the first and the last column
are adjacent too, matching the cyclic nature of a pretzel knot. Consequently, we are able to explicitly describe the linear map $\cA_t(K)=t_1\circ D \circ t_2$, which has a \emph{cyclic tridiagonal structure}: its associated matrix is tridiagonal with non-zero elements at its corners.
\begin{prop}
Let 
\[
D = \operatorname{diag}(D^1,\ldots,D^n)
\]  
be the block diagonal matrix associated with the basic span $\cA_t(T)$ of a pretzel tangle, where each block $D^i$ is the $2\times 2$ matrix corresponding to the basic span of an alternating or codirected braid. Then the span $\cA_t(K)$ is equivalent to the span
\[
0 \;\longleftarrow\; W \;\longrightarrow\; 0,
\]  
where $W$ is described as the kernel of the linear endomorphism with matrix representation  
$$
M = 
\begin{pmatrix}
a_1 & u_1 & 0 & 0 & \cdots & 0 & l_1 \\
l_2 & a_2 & u_2 & 0 & \cdots & 0 & 0 \\
0 & l_3 & a_3 & u_3 & \cdots & 0 & 0 \\
\vdots & \vdots & \vdots & \vdots & \ddots & \vdots & \vdots \\
0 & 0 & 0 & 0 & \ldots &a_{n-1} & u_{n-1}\\
u_n & 0 & 0 & 0 & \cdots & l_n & a_n
\end{pmatrix}.
$$ 
The entries of $M$ are determined by  
$$
a_k = D^{k-1}_{22} - D^{k}_{11}, 
\qquad 
u_k = -D^{k}_{12}, 
\qquad 
l_k = D^{k-1}_{21}, 
\qquad 
k=1,\ldots,n,
$$
where indices are understood cyclically, i.e. modulo $n$.
\end{prop}

\begin{proof}
It is a direct consequence of the block diagonal structure of $\cA_t(T)$ and the pretzel closures $P$ and $Q$.
\end{proof}

Let $S$ be the cycle $\sigma=(1,\ldots,n)$, and let $M_{[i]}$ be the principal minor of $M$ that is obtained when removing the $i$-th row and column of $M$. Since the resulting matrix is tridiagonal, $M_{[i]}$ is the continuant $K^{i}_{n-1}$ defined recursively as in \eqref{eqn:recursioncontinuant}:
\begin{equation} \label{eqn:opencontinuantgeneral}
K^{i}_0=1, \quad K_1^{i}=a_{\sigma^{i}(1)}, \quad K_m^{i}=a_{\sigma^i(m)}K^{i}_{m-1}-u_{\sigma^i(m-1)}l_{\sigma^i(m)}K^{i}_{m-2}, \quad m \geq 2.
\end{equation}
Since $M$ presents the Alexander module of the pretzel knot $K$, we know that any minor of size $n-1$ leads to the Alexander polynomial, since the first elementary ideal is principal. Then
$$
M_{[i+1]}(q_1,\ldots,q_n;t)=M_{[i]}(\sigma(q_1,\ldots,q_n);t),
$$
which also matches the fact that pretzel knots are isotopic under cyclic permutations of the indices, that is,
$$
K(q_1,\ldots,q_n)\sim K(\sigma(q_1,\ldots,q_n)).
$$
We will choose the minor $M_{[1]}$ so that the recursive definition and the first diagonal entry of the resulting continuant $K$ starts involving $q_1$ and $q_2$. Hence
\begin{prop} \label{prop:pretzelAlexrecurrence}
Let $K$ be the pretzel knot $K=K(q_1,\ldots,q_n)$. Let $D=\diag(D^1,\ldots,D^n)$ be the diagonal block matrix defined by the pretzel tangle $T$ such that $\overline{T}=K$. Define
$$
a_k=D^k_{22}-D^{k+1}_{11}, \quad u_k=D^{k+1}_{12}, \quad l_k=-D_{21}^k, \quad k=1,\ldots, n,
$$
where indices are considered modulo $n$. Then the Alexander polynomial $\Delta_{K}(t)$  of $K$ equals $K_{n-1}$, where $\lbrace K_n\rbrace$ are the continuants that are recursively defined as 
$$
K_0=1, \quad K_1=a_1, \quad K_n=a_{n}K_{n-1}-l_nu_{n-1}K_{n-2}, \quad n\geq 2.
$$
\end{prop}
We apply Proposition \ref{prop:pretzelAlexrecurrence} to each of the possible pretzel knots listed in Proposition \ref{prop:pretzeldecomp}.
\subsection{Odd pretzel knots $K=P(2k_1+1,\ldots, 2k_{n}+1)$} In this case, $n$ is odd and every braiding of $K$ is an alternating braiding, and moreover, we may choose orientations so that the first elementary tangle of each braiding is $R^3(X_{\pm})$. Then
$$
\tilde{f}_-^{2k_i+1}(t^{-1})=f_+(t)+k_i(1-t)h, \quad D_i=\begin{pmatrix} (1+k_i)(1-t) & t -k_i(1-t) \\ 1+k_i(1-t) & -k_i(1-t)\end{pmatrix}. 
$$
Applying Proposition  \ref{prop:pretzelAlexrecurrence}, we get that
$$
a_i=(-1-k_i-k_{i+1})(1-t), \quad l_i=-1-k_i(1-t), \quad u_i=t-k_{i+1}(1-t).
$$
Note that $a_i=l_i+u_i$ for all $i$, hence $\Delta_K(t)$ equals
\begin{equation*}
\resizebox{.95\hsize}{!}{$
\det T_{n-1}= \begin{vmatrix} (-1-k_1-k_2)(1-t) & t-k_2(1-t) & 0 &  \ldots & 0 & 0 \\
-1-k_2(1-t) & (-1-k_2-k_3)(1-t) & t-k_3(1-t)  & \ldots & 0 & 0  \\ \vdots & \vdots & \vdots & \ddots & \vdots & \vdots \\ 0 & 0 & 0 & \ldots  & -1-k_{n-1}(1-t) & (-1-k_{n-1}-k_n)(1-t) \end{vmatrix}
$}
\end{equation*}
The structure of the tridiagonal matrix $T_n$ leads us to prove the following result regarding the zeros of its Alexander polynomial.
\begin{thm}
Let $K$ be the alternating odd pretzel knot $K(2k_1+1,\ldots,2k_{n}+1)$, where $n$ is  odd and $k_i\geq 0$ for all $i=1,\ldots, n,$ and let $t_0\in \CC$ be such that $\Delta_K(t_0)=0$. Then $\abs{t_0}=1$.
\end{thm}
\begin{proof}
First, let us write $\delta:=t-1$, $s_i=1+k_i+k_{i+1}$. We may rewrite $l_i,a_i,u_i$ as
$$
a_i=\delta s_i, \quad l_i=k_i\delta-1, \quad u_i=(1+k_{i+1})\delta+1,
$$
so that we may decompose $T_{n-1}$ as $T_{n-1}=\delta N+J$, where $N$ and $J$ are the following tridiagonal matrices with integer coefficients:
$$
N=
\begin{pmatrix}
s_1     & 1+k_2 &        &        &        \\
k_2     & s_2   & 1+k_3  &        &        \\
        & \ddots& \ddots & \ddots &        \\
        &       & k_{n-2}& s_{n-2}& 1+k_{n-1}  \\
        &       &        & k_{n-1}    & s_{n-1}
\end{pmatrix},
\qquad
J=
\begin{pmatrix}
0    & 1    &      &        &      \\
-1   & 0    & 1    &        &      \\
     & \ddots& \ddots& \ddots &      \\
     &       & -1   & 0      & 1    \\
     &       &      & -1     & 0
\end{pmatrix}.
$$
$J$ is clearly skew-symmetric, and we may decompose $N=N^{sym}+N^{skew}$, where $N^{sym}$ is symmetric and $N^{skew}$ is skew-symmetric. Direct computation shows that $N^{skew}=\tfrac{1}{2}J$, and also
$$
N^{sym}_{ii}=s_i, \quad N^{sym}_{i,i+1}=N^{sym}_{i+1,i}=\frac{1+2k_{i+1}}{2}, \quad i=1,\ldots,n-1.
$$
$N^{sym}$ is positive definite: note that $N^{sym}=B^tB+\left( \frac{1+2k_n}{2} \right) e_{n-1}e_{n-1}^t$, where $B$ is the bidiagonal matrix defined as
$$
B_{1,1}=\sqrt{\frac{1+2k_1}{2}}, \qquad B_{i,i}=B_{i,i-1}=\sqrt{\frac{1+2k_{i}}{2}}, \quad i=2,\ldots,n-1.
$$
Therefore, 
$$
T_{n-1}=\delta N^{sym}+\left( 1+\frac{\delta}{2}\right)J.
$$
Now, let $t\in \CC$ such that $\det T_{n-1}(t)=0$. That implies that there exists $v\in \CC^n, v\neq0$ such that 
\begin{equation} \label{eqn:H+1plusdeltamedios}
\delta \, v^{\ast}N^{sym}v+\left( 1+ \frac{\delta}{2} \right)v^{\ast} J v=0.
\end{equation}
Write $h:=v^{\ast}N^{sym}v>0$, $i\gamma:=v^{\ast}Jv$, where $\gamma \in \mathbb{R}$. Let us write $\delta=x+iy$. Taking real and imaginary parts in \eqref{eqn:H+1plusdeltamedios}, we get the system of equations
$$
\begin{cases} xh-\gamma \frac{y}{2}=0, \\ 
yh+\gamma \left(1+\frac{x}{2} \right)=0 \end{cases}
$$
which reduces to $(x+1)^2+y^2=1$, thus $\abs{t}=\abs{\delta+1}=1$.
\end{proof}
\begin{cor}[Hoste]
Let $K$ be the alternating pretzel knot $P(2k_1+1, \ldots, 2k_{n}+1)$, where $n$ is odd and $k_i\geq 0$ for all $i$. Let $t_0\in \CC$ be such that $\Delta_K(t_0)=0$. Then $\Re(t_0)>-1$.
\end{cor}
\begin{proof}
$t_0=-1$ cannot be a root of $\Delta_K=\det T_n$, since then $\delta=-2$ and $T_n=-2N^{sym}$ would be invertible, contradicting $\Delta_K(t_0)=\det T_n=0$.
\end{proof}

We can also use Proposition \ref{prop:pretzelAlexrecurrence} to obtain a closed formula for the Alexander polynomial of an odd pretzel knot, recently proved by Belousov \cite{BelousovAlexander} using the skein relation:
\begin{prop}
Let $K=P(q_1,\ldots,q_n)$ be an odd pretzel knot ($n$ is odd, and $q_i=2k_i+1$ for $i=1,\ldots n$). The Alexander polynomial of $K$ equals:
$$
\Delta_K(t)=\frac{1}{2^{n-1}}\sum_{i\geq 0} (t-1)^{2i}(t+1)^{n-1-2i}\sigma_{2i}(q_1,\ldots,q_n)
$$
where $\sigma_i$ denotes the $i$-th elementary symmetric polynomial.  
\end{prop}
\begin{proof}
Let us write $s_i:=\frac{q_i+q_{i+1}}{2}$, $\alpha_i:=\frac{q_i^2-1}{4}$ and $\delta:=t-1$. The recursion from Theorem \ref{prop:pretzelAlexrecurrence} reads as
\begin{align*}
K_i & =\delta s_i K_{i-1}+\left((1+\delta)-\alpha_{i}\delta^2\right)K_{i-2} \\ & = \delta s_iK_{i-1}+\frac{(t+1)^2}{4}\left( 1-q_i^2\frac{(t-1)^2}{(t+1)^2} \right)K_{i-2}, \quad i\geq 2,
\end{align*}
where we recall that $K_{n-1}=\Delta_K(t)$. If we define $y:=\frac{t-1}{t+1}$, and renormalize $K_i$ as $H_i:=\frac{2^i}{(t+1)^i}K_i$, $H_i$ satisfies the recursion
\begin{equation}
H_i=y(q_i+q_{i-1})H_{i-1}+(1-q_i^2y^2)H_{i-2}, \quad i\geq 2, \label{eqn:Hmrecursion}
\end{equation}
where $H_0=1$, $H_1=y(q_1+q_2)$. Then, we can prove by induction that
\begin{equation}
H_i=\begin{cases} \sum_{k\geq 0} \sigma_{2k}(q_1,\ldots,q_{i+1})y^{2k} & \quad \text{ for even } i, \label{eqn:Hmsymmetric} \\ \sum_{k\geq 0} \sigma_{2k+1}(q_1,\ldots,q_{i+1})y^{2k+1} & \quad \text{ for odd } i. \end{cases}
\end{equation}

We will also write $\sigma_k^{(j)}:=\sigma_k(q_1,\ldots,q_j)$. First, it is immediate to check $H_0=1=\sigma_0$, $H_1= \sigma_1(q_1,q_2)\,y$. Assume $m$ is even, and that \eqref{eqn:Hmsymmetric} holds for $m-1,m-2$. Then, the recursion \eqref{eqn:Hmrecursion} reads as
\begin{align*}
H_m & = y^2(q_m+q_{m-1})\sum_{k\geq 0}\sigma^{(m)}_{2k+1}y^{2k}+(1-q_m^2y^2)\sum_{k\geq 0} \sigma_{2k}^{(m-1)}y^{2k} \\
& = \sum _{r\geq 0} \sigma_{2r}^{(m-1)}y^{2r}+ \sum_{r\geq 0} \left[ (q_m+q_{m-1})\sigma^{(m)}_{2r-1}-q_m^2\sigma_{2r-2}^{(m-1)}     \right] y^{2r} \\& = \left[ \sigma_{2r}^{(m-1)}+(q_m+q_{m-1})\sigma_{2r-1}^{(m-1)}+q_mq_{m-1}\sigma_{2r-2}^{(m-1)} \right]y^{2r} = \sum_{r\geq 0}\sigma_{2r}^{(m+1)}y^{2r},
\end{align*}
after applying the elementary identities $\sigma_{r}^{(m)}=\sigma_{r}^{(m-1)}+q_m\sigma_{r-1}^{(m-1)}$ and $\sigma_{r}^{(m)}=\sigma_{r}^{(m-2)}+(q_{m}+q_{m-1})\sigma_{r-1}^{(m-2)}+q_mq_{m-1}\sigma_{r-2}^{(m-2)}$, being the odd inductive step similar. Hence
$$
\Delta_K(t)=\frac{1}{2^{n-1}}(t+1)^{n-1} H_{n-1}=\frac{1}{2^{n-1}}\sum_{i\geq 0} (t-1)^{2i}(t+1)^{n-1-2i}\sigma_{2i}(q_1,\ldots,q_n).
$$
\end{proof}

\subsection{Even pretzel knots $K(2k_1,2k_2+1,\ldots,2k_n+1)$} Throughout this section, let us write $q_1=2k_1$, $q_i:=2k_i+1$ for $i\geq 2$. In this case, there are two possible braidings according to Proposition \ref{prop:pretzeldecomp}, so we subdivide our computations according to the parity of $n$.

\subsubsection{$n=2p$} All braidings that form the pretzel knot are codirected and their direction alternates: we may choose an orientation for $K$ such that the first elementary tangle of the braid corresponding to $2k_1$ is $X_{\pm}$. As a consequence, for $i=1,\ldots,p,$, Corollary \ref{cor:X+-matrices}
implies that
\begin{equation} \label{eqn:Diforevenpr}
D_{2i-1}=\begin{pmatrix} 1-[q_{2i-1}]_{-t^{-1}} & [q_{2i-1}]_{-t^{-1}} \\ 
t^{-1}[q_{2i-1}]_{-t^{-1}} & 1-t^{-1}[q_{2i-1}]_{-t^{-1}} \end{pmatrix}, \quad  D_{2i}=\begin{pmatrix} 1-[q_{2i}]_{-t} & [q_{2i}]_{-t} \\ 
t[q_{2i}]_{-t} & 1-t[q_{2i}]_{-t} \end{pmatrix},
\end{equation}
thus direct application of Proposition \ref{prop:pretzelAlexrecurrence} yields that
$$
a_{2i-1}=-t^{-1}[q_{2i-1}]_{-t^{-1}}+[q_{2i}]_{-t}, \quad u_{2i-1}=[q_{2i}]_{-t}, \quad l_{2i-1}=-t^{-1}[q_{2i-1}]_{-t^{-1}},
$$
and also
$$
a_{2i}=-t[q_{2i}]_{-t}+[q_{2i+1}]_{-t^{-1}}, \quad u_{2i}=[q_{2i+1}]_{-t^{-1}}, \quad l_{2i}=-t[q_{2i}]_{-t}. 
$$
Note that for all $i=1,\ldots, n$, $a_i=l_i+u_i$. The next proposition provides a closed formula for the Alexander polynomial in terms of quantum integers.
\begin{prop} \label{prop:evenpretzelAlexpolyI}
Let $K$ be the even pretzel knot $P(q_1,q_2,\ldots,q_{2p})$. Then the Alexander polynomial of $K$ equals
\begin{equation} \label{eqn:evenpretzelAlexpolyI}
\Delta_K(t)=\left( \prod_{i=1}^{p} [-q_{2i-1}]_{-t} \prod_{i=1}^{p}[q_{2i}]_{-t}\right) \left( \sum_{i=1}^{p}\frac{1}{[q_{2i}]_{-t}}-\sum_{i=1}^{p} \frac{1}{[-q_{2i-1}]_{-t}} \right) 
\end{equation}
\end{prop}
\begin{proof}
Since $a_i=l_i+u_i$, we may rewrite the continuant recurrence from Theorem  \ref{prop:pretzelAlexrecurrence} as
$$
K_i-u_iK_{i-1}=l_i(K_{i-1}-u_{i-1}K_{i-2}),
$$
so setting $L_i:=K_i-u_iK_{i-1}$, we obtain the linear recurrence
$$
L_i=l_iL_{i-1}, \quad i\geq 2,
$$
so that $L_i=\prod_{j=1}^i l_j$ (note that $L_1=K_1-u_1K_0=a_1-u_1=l_1$). Thus $K_i$ satisfies the affine recurrence relation
$$
K_i=u_iK_{i-1}+L_i, \quad i\geq 1,
$$
or equivalently, writing $E_i=K_i/L_i$, and $r_i=u_i/l_i$,
$$
E_i=r_iE_{i-1}+1,
$$
whose general solution equals
$$
E_i=1+\sum_{j=1}^i\prod_{s=j}^ir_s.
$$
It is straightforward to check that
$$
r_{2s-1}=-t\frac{[q_{2s}]_{-t}}{[q_{2s-1}]_{-t^{-1}}}, \quad r_{2s}=-t^{-1}\frac{[q_{2s+1}]_{-t^{-1}}}{[q_{2s}]_{-t}},
$$
hence for even $i$
$$
\prod_{s=2l+1}^{i} r_s=\frac{[q_{i+1}]_{-t^{-1}}}{[q_{2l+1}]_{-t^{-1}}}, \quad \prod_{s=2l}^{i} r_s=-t^{-1}\frac{[q_{i+1}]_{-t^{-1}}}{[q_{2l}]_{-t}},
$$
and similarly for odd $i$
$$
\prod_{s=2l+1}^{i} r_s=-t\frac{[q_{i+1}]_{-t}}{[q_{2l+1}]_{-t^{-1}}}, \quad \prod_{s=2l}^{i} r_s=\frac{[q_{i+1}]_{-t}}{[q_{2l}]_{-t}}.
$$
Besides,
$$
L_{2p}=\prod_{i=1}^{p} [q_{2i-1}]_{-t^{-1}} \prod_{i=1}^{p}[q_{2i}]_{-t}, \qquad L_{2p-1}=-t^{-1} \prod_{i=1}^{p} [q_{2i-1}]_{-t^{-1}} \prod_{i=1}^{p-1}[q_{2i}]_{-t}.
$$
Setting $i=2p-1$ yields the desired result, as  $\Delta_K(t)=K_{2p-1}$ equals
$$
L_{2p-1}E_{2p-1} = \left( -t^{-1} \prod_{i=1}^{p} [q_{2i-1}]_{-t^{-1}} \prod_{i=1}^{p-1}[q_{2i}]_{-t}\right) \left( 1+[q_{2p}]_{-t}\left( \sum_{i=1}^{p-1}\frac{1}{[q_{2i}]_{-t}}-t\sum_{i=1}^{p} \frac{1}{[q_{2i-1}]_{-t^{-1}}} \right) \right),
$$
which equals \eqref{eqn:evenpretzelAlexpolyI} after applying Properties \eqref{property:quantumint1},\eqref{property:quantumint2}, \eqref{property:quantumint3} of quantum integers.

\end{proof}
Equation \eqref{eqn:evenpretzelAlexpolyI} allows us to prove that the zeros of the Alexander polynomial are bounded when $k_i>0$. We prove first an elementary lemma:
\begin{lem}\label{lem:complexineq}
Let $z\in \CC$ such that $\abs{z}>1$. Then $\Re\left(\frac{1}{z-1}\right)>-\frac{1}{2}$ 
\end{lem}
\begin{proof}
Consider $\phi:\CC \rightarrow \CC$ defined as $\phi(z)=\frac{z+1}{z-1}=1+\frac{2}{z-1}$. Then $\Re(\phi(z))>0$ if $\abs{z}>1$, so $\Re(\frac{1}{z-1})=\Re(\frac{\phi(z)-1}{2})>-\frac{1}{2}$.
\end{proof}
\begin{thm} \label{thm:HosteevenpretzelI}
Let $K$ be the alternating even pretzel knot $P(q_1,q_2,\ldots,q_{2p})$, where $q_i>0$, and let $t\in \CC$ be such that $\Delta_K(t)=0$. Then $\abs{t}= 1$.   
\end{thm}
\begin{proof}
Assume that $t\in \CC$ is such that $\abs{t}>1$, $\Delta_K(t)=0$. Then, it is clear that $$
L_{2p-1}=\left( \prod_{i=1}^{p} [-q_{2i-1}]_{-t} \prod_{i=1}^{p}[q_{2i}]_{-t}\right) \neq 0,
$$
since the only roots of $[\pm q_i]_{-t}$ are roots of unity, hence $E_{2p-1}(t)=0$. Thus
\begin{equation}\label{eqn:quantevenpretzelhoste}
\sum_{i=1}^{p}\frac{1}{[q_{2i}]_{-t}}-\sum_{i=1}^p
\frac{1}{[-q_{2i-1}]_{-t}}=0.
\end{equation}
Setting $z:=-t$ and applying elementary properties of quantum integers, \eqref{eqn:quantevenpretzelhoste} holds for $z\neq 1$ if and only if
\begin{equation} \label{eqn:moduluseqpret}
\sum_{i=i}^{2p} \frac{1}{z^{q_i}-1}+p=0.
\end{equation}
But direct application of \ref{lem:complexineq} shows that the real part of the left hand side of \eqref{eqn:moduluseqpret} is strictly positive, hence a contradiction. Since $\Delta_K(t)=\Delta_K(t^{-1})$ up to a power of $t$, this also excludes the case where $\abs{t}<1$.
\end{proof}
As the determinant of the knot is a non-zero odd integer, $t=-1$ is not a root of the Alexander polynomial:
\begin{cor}[Hoste]
Let $K$ be the alternating pretzel even knot $P(q_1,q_2,\ldots,q_{2p})$, $q_i> 0$. Then $\Re(t)>-1$.
\end{cor}

\subsubsection{$n=2p+1$} In this case, the first braiding is alternating, and the others are codirected. Proceeding as before, we may orient $K$ so that the first braiding is exactly $\widetilde{T}_+^{2k_1}$, 
$$
D_1=\id -k_1(1-t^{-1})h=\begin{pmatrix} 1-k_1(1-t^{-1}) & k_1(1-t^{-1}) \\-k_1(1-t^{-1}) & 1+k_1(1-t^{-1})\end{pmatrix},
$$
where we recall that $q_1=2k_1$. $D_i$, for $i\geq 2$, matches exactly the pattern from the case $n=2p$ in \eqref{eqn:Diforevenpr}. As a consequence, the same recursion is satisfied, with the sole exception of the initial condition that equals
$$
a_1=k_1[2]_{-t^{-1}}+[q_2]_{-t}, \quad u_1=[q_2]_{-t}, \quad  l_1=k_1[2]_{-t^{-1}}.
$$
 The same strategy yields an analogous formula for even pretzel knots with and odd number of strands, as the next proposition shows.
\begin{prop}
Let $K$ be the even pretzel knot $P(q_1,q_2,\ldots,q_{2p+1})$. The Alexander polynomial of $K$ equals
\begin{equation}\label{eqn:evenpretzelAlexpolyII}
\Delta_K(t)  =\left( \frac{q_1}{2}[-2]_{-t} \prod_{i=2}^{p+1} [-q_{2i-1}]_{-t} \prod_{i=1}^{p}[q_{2i}]_{-t}\right)  \left( \sum_{i=1}^{p} \frac{1}{[q_{2i}]_{-t}}-\sum_{i=2}^{p+1}\frac{1}{[-q_{2i-1}]_{-t}}+\frac{2}{q_1}\frac{1}{t[-2]_{-t}} \right).  
\end{equation}
\end{prop}
\begin{proof}
The proof from Proposition \ref{prop:evenpretzelAlexpolyI} works verbatim, since the recurrence for $E_i=D_i/L_i$ still holds for $i\geq 2$ and all terms are equal. The change of initial conditions modifies the following terms
$$
u_1=[q_2]_{-t}, \quad l_1=k_1[2]_{-t^{-1}}, \quad r_1=\frac{[q_2]_{-t}}{k_1[2]_{-t^{-1}}},
$$
so that 
$$
L_{2p}=-k_1t[2]_{-t^{~1}}\left( \prod_{i=2}^{p}[q_{2i-1}]_{-t^{-1}} \right)\left( \prod_{i=1}^{p} [q_{2i}]_{-t} \right), \quad \prod_{s=1}^{2p}r_s=\frac{-t^{-1}[q_{2p+1}]_{-t^{-1}}}{k_1[2]_{-t^{-1}}},
$$
and recalling that $E_i=1+\sum_{j=1}^i\prod_{s=j}^ir_s$,
$$
L_{2p}E_{2p}  =\left( -\frac{q_1}{2}t[2]_{-t^{-1}} \prod_{i=2}^{p+1} [q_{2i-1}]_{-t^{-1}} \prod_{i=1}^{p}[q_{2i}]_{-t}\right)  \left( \sum_{i=2}^{p+1}\frac{1}{[q_{2i-1}]_{-t^{-1}}}-t^{-1}\sum_{i=1}^{p} \frac{1}{[q_{2i}]_{-t}}-\frac{2}{q_1}t^{-1}\frac{1}{[2]_{-t^{-1}}} \right)  
$$
and the result follows.
\end{proof}
In this case, it is not longer true that any root $t$ of the Alexander polynomial of $K$ satisfies $\abs{t}\leq 1$. For example, the Alexander polynomial of the alternating knot $9_{11}=P(-5,2,1,1,1)$ equals
$\Delta_K(t)=1-5t+7t^2-7t^3+7t^4-5t^5+t^6$, which has a positive real root in $[3,4]$. We provide an alternative proof to \cite{IshikawaMontesinos} of Hoste's bound for odd $n$ for this family:
\begin{thm}[Hoste] \label{thm:HosteevenpretzelII}
Let $K$ be the alternating even pretzel knot $P(q_1,q_2,\ldots,q_{2p+1})$, where $q_i>0$ and let $t\in \CC$ be such that $\Delta_K(t)=0$. Then $\Re(t)>-1$.
\end{thm}
\begin{proof}
Let $t\in \CC$ such that $\Re(t)<-1$. Proceeding as in Theorem \ref{thm:HosteevenpretzelI}, it suffices to show that 
\begin{equation} \label{eqn:hosteeqnII}
\Re \left( \sum_{i=1}^{p} \frac{1}{[q_{2i}]_{-t}}-\sum_{i=2}^{p+1}\frac{1}{[-q_{2i-1}]_{-t}}+\frac{2}{q_1}\frac{1}{t[-2]_{-t}} \right)\neq 0,
\end{equation}
which proves that $t$ is not a root of the Alexander polynomial. Writing $z=-t$, the left hand side of \eqref{eqn:hosteeqnII} vanishes if and only if
$$
p+\sum_{i=2}^{2p+1}\frac{1}{z^{q_i}-1}+\frac{2z}{k_1(z^2-1)}=p+\sum_{i=2}^{2p+1}\frac{1}{z^{q_i}-1}+\frac{1}{q_1}\left(\frac{1}{z-1}+\frac{1}{z+1}\right)
$$
Lemma \ref{lem:complexineq} yields that $\Re\left(p+\sum_{i=2}^{2p+1} \frac{1}{z^{q_i}-1} \right)>0$, but also $\frac{1}{z-1}, \,\frac{1}{z+1}$ have positive real parts since $\Re(z)>1$.
\end{proof}
\section{The functor $\cA_{-1}$} \label{sec:functorA-1}

The functor $\cA_{-1}$ is related to several classical invariants of knots and encodes several aspects of their arithmetic and topology. For instance, it is a classical result that evaluating the Alexander polynomial at $t=-1$ yields the determinant of $K$, defined as the order of $H_1(X_2)$, where $X_2$ denotes the double cover of $S^3$ branched over $K$. Another well-known case where the specialization $t=-1$ recovers an important invariant arises in the theory of Alexander linear quandles: when $t=-1$, the quandle operation on $\ZZ_p$ reduces to
$$
a \triangleright_{-1} b = 2a - b,
$$
which is classically connected to knot colorings \cite{ColoringsKauffman}.

In our setting, $t=-1$ is a fixed point of the involution $\tau:\CC^{\ast}\to\CC^{\ast}$, $\tau(t)=t^{-1}$. Geometrically, Remark~\ref{rem:algebraicinvolution} shows that the fibres of $\cA(T)$ and $\cA(T')$ at $t=-1$ are isomorphic, a fact which also follows algebraically from Corollary~\ref{cor:spanreverse}. Consequently, $\cA_{-1}(T)$ depends only on the unoriented isotopy class of the tangle, which we denote by $[T]$. From a categorical perspective, the functor $\cT_{-1}$ factors through the forgetful functor $\cT \to \cT^{\mathrm{unor}}$, where $\cT^{\mathrm{unor}}$ denotes the category of unoriented tangles. Thus:

\begin{lem} There is a well-defined monoidal functor
$\cA_{-1}: \cT^{unor}\longrightarrow \Span(\Vect_\CC)$.
\end{lem}

Moreover, our computations show that $t=-1$ is the unique value for which the maps $f_{\pm} : \CC^2 \to \CC^2$ fail to be diagonalizable, as the spectrum of $f_{\pm}$ is given by $\sigma=\lbrace1,-t \rbrace$. In addition, since $[n]_{-1}=n$, Propositions \ref{prop:powersf+f-} and \ref{prop:twistedtanglesformulasII} specialize to
\begin{align*}
f^n_{+}(a_1,a_2) & =((1+n)a_1-na_2, na_1+(1-n)a_2)=\tilde{f}_-^n(a_1,a_2) \\f^n_{-}(a_1,a_2) & =((1-n)a_1+na_2, -na_1+(1+n)a_2)=\tilde{f}_+^n(a_1,a_2),
\end{align*}
which leads us to the following definition.
\begin{defn}
For $s\in \CC$, let $f^s$ be the linear endomorphism
$$
f^s: \CC^2 \longrightarrow \CC^2, \quad f^s=\id+sh,
$$
where $h:\CC^2\longrightarrow\CC^2$ is the nilpotent map given by $h(a_1,a_2)=(a_1-a_2,a_1-a_2)$.
\end{defn}

It is immediate to verify that $f^{s'} \circ f^s= f^{s'+s}$ and that $f=f_+, \,f^{-1}=f_-$ . As we did before, we will in this section identify a linear map with a basic span via the embedding of $\Vect_{\CC}$ inside $\Span(\Vect_{\CC})$. A direct computation of the rotation of $f^s$ shows that they behave nicely with respect to rotations:
\begin{prop}\label{prop:rotationfs}
The rotation of the basic span $f^s$ is equivalent to the basic span $f^{-1/s}$
\end{prop}
\begin{proof}
Applying Lemma \ref{lem:rotation2tangle}, the rotation of $f^s$ equals
$$
\xymatrix{
	& \CC^{2} \ar[ld]_{g_1} \ar[rd]^{g_2} & \\
	\CC^2 & & \CC^{2}
} 
$$
where $g_1(a_1,a_2)=((1+s)a_1-sa_2,a_1)$ and $g_2(a_1,a_2)=(sa_1+(1-s)a_2,a_2)$, equivalent to the basic span determined by $g_2\circ g_1^{-1}=f^{-1/s}$.
\end{proof}

As a consequence, following the conventions from \cite{kauffmanrationaltangles}, where $[X_+]=[-1]$ and $[X_-]=[+1]$,  we get that $\cA_{-1}([n])=f^{-1/n}$ for $n\in \ZZ$. Thus
\begin{cor} \label{cor:powersandrotat-1} For any $n\in \ZZ$,
$\cA_{-1}(\pm1/[n])=f^{\mp n}, \quad \cA_{-1}(\pm [n])=f^{\mp1/n}$
\end{cor}
For any rational number $p/q\in \mathbb{Q}$, there exists a positive standard continued fraction $[a_1,\ldots,a_n]$, which is unique if all $a_i$ have the same sign and $n$ is odd. In analogy, rational tangles admit a continued fraction description: every rational tangle is isotopic to a tangle obtained by iterated vertical compositions of the elementary tangles $[X_{\pm}]$ on the top, together with successive horizontal additions of $[X_{\pm}]$ on the right. Moreover, there exists a \emph{canonical} continued fraction form of the tangle, characterized by the property that its planar projection is alternating and the number of iterations is odd, topologically mirroring the arithmetic properties of continued fractions.  

We restate this theorem from \cite{kauffmanrationaltangles} in terms of rotations and compositions:
\begin{thm} \label{thm:continuousfractangle}
Any rational tangle is isotopic to a tangle in canonical form, that is
\begin{equation}\label{eqn:tanglecontinuedform}
T_{[a_1,\ldots,a_n]}\sim R^{-1}([X_+]^{a_1}\circ R([X_-]^{a_{2}}\circ \ldots R^{-1}([X_+]^{a_{n-2}}\circ R([X_-]^{a_{n-1}}\circ R^{-1}([X_+]^{a_n})))))
\end{equation}
where $a_i \in \mathbb{N}$ and $n$ is odd.
\end{thm}

\begin{thm}  \label{thm:rationaltanglesspan}
Given any rational tangle $T_{p/q}$, $\mathcal{A}_{-1}(T)$ is the basic span determined by the map $f^{-q/p}$.
\end{thm}
\begin{proof}
Let $T_{p/q}$ be a rational tangle, with canonical continued fraction expansion $[a_1,\ldots,a_n]$. By Theorem \ref{thm:continuousfractangle}, we may assume that T is in continued fraction form, as in \eqref{eqn:tanglecontinuedform}. Let us show by induction that $A_{-1}(T_{[a_1,\ldots, a_n]})$ is the basic span $f^{-1/[a_1,\ldots,a_n]}$. For $n=1$ the assertion is clear. Assuming that the statement holds for any tangle whose continued fraction expansion has length $n-1$, we can apply Proposition \ref{prop:rotationfs} so that
$$
\cA_{-1}([X_-]^{a_2}\circ T_{[a_3,\ldots,a_n]})=f^{-[a_2,\ldots,a_n]},
$$
hence
$$
\cA_{-1}(R(T_{[a_1,\ldots,a_n]})) = \cA_{-1}([X_+]^{a_1}\circ  R([X_-]^{a_2}\circ T_{[a_3,\ldots,a_n]})) = f^{[a_1,\ldots,a_n]}.
$$
\end{proof}

\begin{cor}
The functor $\cA_{-1}$ classifies rational tangles up to isotopy.
\end{cor}

\begin{rem}
Theorem \ref{thm:rationaltanglesspan} shows that the specialization at $t=-1$ of the functor $\cA_{t}$ classifies rational tangles. In recent years, a notion of quantized rationals has been studied from an arithmetic perspective \cite{MorierGenoudOvsienko}, connected with the Jones polynomial \cite{SikoraJones}. We expect that a generalization of Theorem \ref{thm:rationaltanglesspan} to $\cA(T)$ will yield a geometric interpretation of quantized rationals, a perspective we intend to explore in future work.
\end{rem}

\subsection{Colorings} The fraction $p/q$ associated with a rational tangle can be obtained via colorings. We briefly summarize the relevant results and definitions here, referring to \cite{kauffmanrationaltangles} for a detailed exposition and complete proofs.

Given any knot $K$ represented in terms of a planar diagram, an \emph{(integral) coloring of $K$} is a function from the set of arcs of $K$ to $\mathbb{Z}$, such that Fox's coloring equations 
\begin{equation} \label{eqn:Foxcoloreq}
2a=b+c,
\end{equation}
are satisfied at any crossing, where $a$ denotes the color of the overcrossing arc and $b,c$ are the colors of the two undercrossing arcs.

\begin{figure}[h]
\includegraphics[height=5cm]{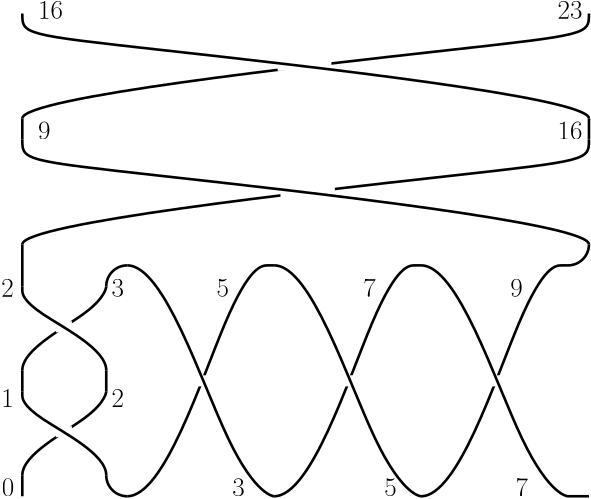}
\caption{An integral coloring of the tangle $7/16=[0;2,3,2]$} \label{fig:Fox322}
\end{figure}

Unlike knots, where (non-trivial) colorings must sometimes be reduced to $\mathbb{Z}_n$ in order to exist, rational tangles are integrally colorable: one can choose colors for the initial strands of any rational knot (where the tangle starts to wind), so that the coloring propagates to the rest of the tangle.

To any coloring of a rational tangle $T$, as a 2-tangle, we can associate its \emph{color matrix},
$$
M(T)=\begin{pmatrix} a & b \\c & d \end{pmatrix},
$$
where $a,b,c,d \in \mathbb{Z}$ are the colors of the NW, NE, SW and SE external strands of $T$, respectively. For example, the color matrix of the coloring displayed in Figure \ref{fig:Fox322} is $M=\left( \begin{smallmatrix}
    16 & 23 \\ 0 & 7
\end{smallmatrix} \right)$.

We also recall the following properties of any color matrix of $T$:
\begin{itemize}
    \item $a-b=c-d$ (defined in \cite{kauffmanrationaltangles} as \emph{the diagonal sum rule})
    \item For any $k,l\in \ZZ$, $\left( \begin{smallmatrix} ka+l & kb+l \\ kc+l & kd+l\end{smallmatrix} \right) $ is also a coloring matrix for a rational tangle $T$.
\end{itemize}
Thus any coloring of $T$ defines a coloring matrix $M(T)$. Conversely, by the propagation property, any coloring matrix $M(T)$ uniquely recovers the full coloring on the tangle $T$. For this reason, we shall occasionally refer to $M(T)$ itself as a coloring of $T$. The coloring invariant that can be associated to any rational tangle $T$ is the rational number
$$
f(T)=\frac{b-a}{b-d} \in \mathbb{Q},
$$
which classifies the rational tangle $T$ up to isotopy. It coincides with the ``topological'' fraction that can be associated to any rational tangle from its continued fraction form. 

Colorings are classically related to Alexander quandles \cite{ColoringsKauffman, cattabriga}, since Fox coloring equation \eqref{eqn:Foxcoloreq} is the specialization at $t=-1$ of the Alexander quandle relation \eqref{eqn:Alexquandlerel} at every crossing.
The goal of this section is to provide a complex-geometrical interpretation in terms of spans of the fraction $p/q$ of a rational tangle, relating colorings and the functor $\cA_{-1}$. We first restate and prove some of the previous properties and statements in terms of representations, quandles and the functor $\cA_{-1}$.

In what follows, $T$ is assumed to be a rational tangle, and $\partial T=\epsilon \sqcup \varepsilon$.

First of all, $\R_{\AQ{-1}}(\epsilon)\cong R_{\AQ{-1}}(\varepsilon)\cong \CC^2$, as $Q_{\epsilon}\cong Q_{\varepsilon}$ is the free quandle in two generators. $\R_{\AQ{-1}}(T)$ is also two-dimensional for any rational tangle $T$:
\begin{prop}\label{prop:2dimcolorings}
For any rational tangle $T$, $\R_{\AQ{-1}}(T)\cong \CC^2$
\end{prop}
\begin{proof}
Choosing an alternating continued fraction planar diagram for $T$, either the inclusion $i:\epsilon^c\hookrightarrow T^c$ or $i: (\epsilon')^ c\hookrightarrow R(T)^c$ is a homotopy equivalence, where $R(T)=\epsilon'\sqcup \varepsilon'$ is the rotation of the tangle $T$. Hence $\R_{\AQ{-1}}(T)\cong \R_{\AQ{-1}}(\epsilon)\cong \CC^2$ or $\R_{\AQ{-1}}(T)\cong \R_{\AQ{-1}}(R(T))\cong \R_{\AQ{-1}}(\epsilon')\cong \CC^2$.
\end{proof}

\begin{rem}
Taking into account the trivial coloring $\left( \begin{smallmatrix} 1 & 1 \\ 1 & 1 
\end{smallmatrix} \right) $ that corresponds to the trivial representation, the space of colorings is generated by any single non-trivial coloring. Moreover, any coloring on any two consecutive points $\tau:=\lbrace p_i,p_j \rbrace $ of $\partial T$ propagates to a coloring of $T$ if and only if the inclusion of $\tau^c$ inside $T^c$ is a (weak) homotopy equivalence, which implies that $\R_{\AQ{1}}(\tau)\cong \R_{\AQ{1}}(T)\cong \CC^2$.
\end{rem}

The \emph{diagonal sum rule} property of colorings of rational tangles can also be interpreted in terms of representations, as the next proposition shows.

\begin{prop}[Diagonal sum rule] \label{prop:diagonalsumrule} Let $T$ be a rational tangle and let $\cA_{-1}(T)$ be the span
$$
\xymatrix{
	& \CC^{2} \ar[ld]_{f} \ar[rd]^{g} & \\
	\CC^{2} & & \CC^{2}
} 
$$
Then, for any $v\in\CC^2$, if we write $f(v)=(a,b)\in \CC^2$, $g(v)=(c,d)\in \CC^2$,
$$
a-b=c-d
$$
\end{prop}
\begin{proof} First, let us write $\pi_1(\varepsilon^c)=\langle \alpha_1,\alpha_2\rangle \cong \mathbf{F}_2$ and $\pi_1(\epsilon^c)=\langle \beta_1,\beta_2\rangle \cong \mathbf{F}_2$, where each generator is a simple loop around each of the punctures. Recalling the adjunction \eqref{eqn:adjunctiongrqdl}, choosing $v\in \CC^2\cong \R_{\AQ{-1}}(Q_T)$ is equivalent to choosing $\rho\in \Hom(\pi_1(T^c), \aglC)$ such that $\pi_K(\rho)=-1$. Writing $a,b,c,d$ for the off-diagonal terms of the images by $\rho$ in $\aglC$ of the four generators $\alpha_1,\alpha_2,\beta_1,\beta_2$, we get that
$$
\begin{pmatrix} 1 & 0 \\ a-b & 1 \end{pmatrix} = \rho(i(\alpha_1))\rho(i(\alpha_2))=\rho(i(\alpha_1\ast \alpha_2))=\rho(i(\beta_1\ast \beta_2))=\rho(i(\beta_1))\rho(i(\beta_2))=\begin{pmatrix} 1 & 0 \\ c-d & 1 \end{pmatrix},
$$
since the inclusions of the compositions $\alpha_1\ast\alpha_2$ and $\beta_1\ast \beta_2$ are homotopic loops in $T^c$.
\end{proof}

The following proposition clarifies its relation with the functor $\cA_{-1}(T)$.
\begin{prop}
Any non-trivial coloring of a rational tangle $T$ is completely determined by the span $\cA_{-1}(T)$
$$
\quad \xymatrix{
	& \CC^{2} \ar[ld]_{f} \ar[rd]^{g} & \\
	\CC^{2} & & \CC^{2}
} 
$$
and viceversa. Moreover, the matrix
$$
M_v= \left(\begin{array}{c} g(v) \\ \hline f(v) \end{array} \right)\quad v\in \CC^2
$$
is a coloring matrix for $T$.
\end{prop}
\begin{proof} 
Proposition \ref{prop:2dimcolorings} yields that the three vector spaces that define the span $\cA_{-1}(T)$ are isomorphic to $\mathbb{C}^2$. Moreover, Fox's equations imply that $f(1,1)=g(1,1)=(1,1)\in \CC^2$, as the trivial coloring in $T$ yields the trivial coloring on $\partial T$. Any non-trivial coloring of $T$ amounts to choosing $v\in \R_{\AQ{-1}}(T), v\not\in \langle (1,1) \rangle$, which is uniquely determined by the induced coloring on the endpoints of the tangle, that is, $\left( \begin{array}{c c} a & b \\c & d \end{array} \right):=\left( \begin{array}{c } g(v) \\ \hline f(v) \end{array} \right)$ and viceversa. This information fully characterizes the span $\cA_{-1}(T)$ and the linear maps $f,g:\CC^2\rightarrow\CC^2$, as $\mathcal{B}=\lbrace (1,1),v\rbrace $ is a basis for $\CC^2$.
\end{proof}
As a consequence, the span $\cA_{-1}(T)$ can be regarded as a monomorphism $\phi_T:=(f,g):\CC^2 \longrightarrow \CC^4$ that satisfies that $\phi_T(1,1)=(1,1,1,1)$. The image of $\phi_T$ is a plane in $\CC^4$ generated by $v_0:=(1,1,1,1)$ and any non-trivial coloring $v:=(a,b,c,d)$ on $\partial{T}$. Note that the choice of any equivalent span to $\cA_{-1}(T)$ corresponds to a different choice of basis for $\R_{\AQ{-1}}(T)\cong \CC^2$, but the plane defined by the image of $ \phi_T$ remains invariant. That is,
\begin{cor}\label{cor:spaningrassmannian}
$\cA_{-1}(T)$ uniquely defines a point $P_T\in \Gr(2,4)$
\end{cor}
Let us recall that the complex grassmannian $\Gr(2,4)$ can be realized as a complex algebraic subvariety of $\mathbb{P}^5(\CC)$ via the Pl{\"u}cker embedding, that maps any $P=\langle u,v\rangle \in \Gr(2,4)$ to the line generated by the bivector $u\wedge v$ in $\bigwedge^2 \mathbb{C}^4$, a point in $\mathbb{P}(\bigwedge^2 \mathbb{C}^4) \cong \mathbb{P}^5(\mathbb{C})$. Its image in $\mathbb{P}^5(\mathbb{C})$ is the Pl{\"u}cker quadric $\mathcal{C}$ defined by the equation
$$
p_{12}p_{34}-p_{13}p_{24}+p_{14}p_{23}=0,
$$
where $p_{ij}=\begin{vmatrix} v_i & v_j \\w_i & w_j \end{vmatrix}$ are the standard coordinates of $\bigwedge^2 \CC^4$. 

We are interested in describing the set
$$
\cA_{-1}^{rat}:=\lbrace \cA_{-1}(T) : T \text{ is a rational tangle} \rbrace,
$$ which by Corollary \ref{cor:spaningrassmannian} can be understood as the image (via colorings) of all rational tangles inside $G(2,4)$, thanks to the functor $\cA_{-1}$. First of all, let us consider the Schubert variety $\sigma \subset \mathcal{C}$ defined as
$$
\sigma = \lbrace  P\in \mathcal{C} \mid \dim( \langle v_0\rangle \cap P ) \geq 1\rbrace= \lbrace \langle v_0,v\rangle \mid v\in \CC^4 \rbrace \subset \mathcal{C} \subset \Gr(2,4),
$$
which consists of planes in $\CC^4$ that contain the line generated by $v_0$. We note that in this case $p_{ij}= \begin{vmatrix}
    1 & 1 \\ v_i & v_j
\end{vmatrix}=v_j-v_i$, so that $\sigma$ is defined by the three linear equations on $\Gr(2,4)$
\begin{align}
p_{12}-p_{13}+p_{23} & = 0, \label{eqn:plucker1} \\
p_{12}-p_{14}+p_{24} & = 0, \label{eqn:plucker2}\\
p_{13}-p_{14}+p_{34} & =0.\label{eqn:plucker3}
\end{align}
Thus $\sigma \cong \mathbb{P}^2(\CC)$, parametrized by $[p_{12}:p_{13}:p_{14}]$.

Furthermore, Proposition \ref{prop:diagonalsumrule} establishes that any point $P_T$ in $G(2,4)$, which arises from a rational tangle $T$, satisfies the additional equation $v_1-v_2=v_3-v_4$. This imposes the additional linear equation $p_{12}=p_{34}$ on $\sigma$, equivalent to
\begin{equation} \label{eqn:plucker4}
p_{12}+p_{13}-p_{14}=0.
\end{equation}
We summarize this discussion in the following theorem.
\begin{thm}
The set $\cA_{-1}^{rat}:=\lbrace \cA_{-1}(T) \mid T \text{ is a rational tangle} \rbrace$ defines a (smooth) rational curve inside the complex grassmannian $\Gr(2,4)$. 
The fraction associated to any coloring $(v_1,v_2,v_3,v_4)$ of a rational tangle $T$, 
$$
\frac{p}{q}=\frac{v_4-v_3}{v_4-v_2},
$$
is the rational slope of the point $\cA_{-1}(T)$ inside $\cA_{-1}^{rat}\cong \mathbb{P}^1(\CC)$.
\end{thm}
\begin{proof}
Equations \eqref{eqn:plucker1}-\eqref{eqn:plucker4} define $\cA_{-1}^{rat}$ as a complex line inside $\mathcal{C}\subset \mathbb{P}^5(\CC)$, that can be explictly described using Pl{\"u}cker coordinates via the embedding
\begin{align*}
\Phi \colon  \mathbb{P}^1(\mathbb{C}) & \longrightarrow \mathbb{P}^5(\mathbb{C}) \\
[p_{12}:p_{13}] & \longmapsto [p_{12}:p_{13}:p_{12}+p_{13}:-p_{12}+p_{13}:p_{13}:p_{12}]
\end{align*}
Besides, note that these equations also imply that the rational number
$$
\frac{p}{q}:=\frac{v_4-v_3}{v_4-v_2}=\frac{p_{34}}{p_{24}}=\frac{p_{12}}{p_{13}},
$$
is simply the (rational) slope of the point $[p_{12}:p_{13}]\in \mathbb{P}^1(\CC)$, regarded as a line inside the affine chart $\lbrace p_{13}\neq 0\rbrace \cong \CC$.
\end{proof}
\bibliographystyle{abbrv}
\bibliography{AlexanderRational}
\end{document}